\documentstyle[11pt,leqno]{article}

\include{psfig}
\include{epsf}

\newtheorem{theorem}{Theorem}[section]
\newtheorem{lemma}{Lemma}[section]

\newtheorem{algorithm}{Algorithm}[section]
\newtheorem{remark}{Remark}[section]

\catcode`\@=11
\renewcommand{\section}{
         \setcounter{equation}{0}
         \@startsection {section}{1}{\z@}{-3.5ex plus -1ex minus
         -.2ex}{2.3ex plus .2ex}{\normalsize\bf}
}
\renewcommand{\subsection}{
         \@startsection {subsection}{1}{\z@}{-3.5ex plus -1ex minus
         -.2ex}{2.3ex plus .2ex}{\normalsize\bf}
} \catcode`\@=12

\def\reals{{\rm\vrule depth0ex width.4pt\kern-.08em R}}
\def\bbbz{{\mathchoice {\hbox{$\sf\textstyle Z\kern-0.4em Z$}}
{\hbox{$\sf\textstyle Z\kern-0.4em Z$}} {\hbox{$\sf\scriptstyle
Z\kern-0.3em Z$}} {\hbox{$\sf\scriptscriptstyle Z\kern-0.2em Z$}}}}

\newcommand{\nc}{\newcommand}
\nc{\W}{{\bf W}} \nc{\A}{{\bf A}} \nc{\bL}{{\bf L}} \nc{\bH}{{\bf
H}} \nc{\C}{{\cal C}}

\def\eq#1{(\ref{e:#1})}

\def\elabel#1{\label{e:#1}}

\begin{document}
\begin{center}
\Large\bf Numerical Methods and Analysis via Random Field Based
Malliavin Calculus for Backward Stochastic PDEs~\footnote{This paper
was presented in 2013 IMA Theory and Applications of Stochastic PDEs
from January 13-18 of 2013 in Minneapolis, Minnesota, U.S.A. and
will be presented as an invited talk in IMS-China 2013 from June 30
to July 4 of 2013 in Chengdu, China.}
\end{center}
\begin{center}
\large\bf Wanyang Dai~\footnote{Supported by National Natural
Science Foundation of China under grant No. 10971249.}
\end{center}
\begin{center}
\small Department of Mathematics and State Key Laboratory of Novel
Software Technology\\
Nanjing University, Nanjing 210093, China\\
Email: nan5lu8@netra.nju.edu.cn\\
Date: 28 June 2013
\end{center}

\vskip 0.1 in
\begin{abstract}
We study the adapted solution, numerical methods, and related
convergence analysis for a unified backward stochastic partial
differential equation (B-SPDE). The equation is vector-valued, whose
drift and diffusion coefficients may involve nonlinear and
high-order partial differential operators. Under certain generalized
Lipschitz and linear growth conditions, the existence and uniqueness
of adapted solution to the B-SPDE are justified. The methods are
based on completely discrete schemes in terms of both time and
space. The analysis concerning error estimation or rate of
convergence of the methods is conducted. The key of the analysis is
to develop new theory for random field based Malliavin calculus to
prove the existence and uniqueness of adapted solutions to the
first-order and second-order Malliavin derivative based B-SPDEs
under random environments.\\

\noindent{\bf Key words and phrases:} Numerical Method, Error
Estimation, Convergence, Random Field Based Malliavin Calculus,
Backward SPDE, High-Order
Partial Differential Operator, Nonlinear, Random Environment\\
\noindent{\bf Mathematics Subject Classification 2000:} 60H35,
65C30, 60H15, 60K37, 60H30
\end{abstract}

\section{Introduction}

In this paper, we study the adapted solution, numerical schemes, and
related convergence analysis for a unified backward stochastic
partial differential equation (B-SPDE) given by
\begin{eqnarray}
\;\;\;\;\;V(t,x)&=&H(x)+\int_{t}^{T}{\cal L}(s,x,V)ds
+\int_{t}^{T}\left({\cal J}(s,x,V) -\bar{V}(s,x)\right)dW(s).
\elabel{bspdef}
\end{eqnarray}
The equation in \eq{bspdef} is vector valued. The nonlinear partial
differential operators ${\cal L}$ and ${\cal J}$ depend not only on
$V$ and/or $\bar{V}$ but also on their associated high-order partial
derivatives, e.g., up to the $k$th, $m$th, and $n$th orders for
$k,m,n\in\{0,1,2,...\}$,
\begin{eqnarray}
{\cal L}(s,x,V)&\equiv&{\cal
L}(s,x,V(s,x),...,V^{(k)}(s,x),\bar{V}(s,x),...,
\bar{V}^{(m)}(s,x)),\nonumber\\
{\cal J}(s,x,V)&\equiv&{\cal J}(s,x,V(s,x),...,V^{(n)}(s,x)).
\nonumber
\end{eqnarray}
The unified B-SPDE in \eq{bspdef} covers many existing systems as
special cases. For examples, when ${\cal J}=0$ and ${\cal L}$
depends only on $x,V,\bar{V}$, but not on their associated partial
derivatives, it reduces to a conventional backward (ordinary)
stochastic differential equation (BSDE) (see, e.g., Pardoux and
Peng~\cite{parpen:adasol}); Furthermore, when ${\cal J}=0$ and
${\cal L}$ depends on both the derivatives of $V$ and $\bar{V}$, it
reduces to a well-known example of strongly nonlinear B-SPDE derived
in Musiela and Zariphopoulou~\cite{muszar:stopar} for the purpose of
optimal-utility based portfolio choice. We here note that the
strongly nonlinearity concerning the operator ${\cal L}$ is in the
sense addressed in Lions and Souganidis~\cite{liosou:notaux},
Pardoux~\cite{par:stopar}. Besides these existing examples, our
motivations to study the B-SPDE in \eq{bspdef} are also from optimal
portfolio management in finance (see, e.g.,
Becherer~\cite{bec:bousol}, Dai~\cite{dai:meavar,dai:opthed},
Musiela and Zariphopoulou~\cite{muszar:stopar}), and multi-channel
(or multi-valued) image regularization such as color images in
computer vision and network application (see, e.g., Caselles {\em et
al.}~\cite{cassap:vecmed}, Tschumperl\'e and
Deriche~\cite{tscder:regort,tscder:conunc,tscder:anidif}).

Under certain generalized Lipschitz and linear growth conditions, we
adopt a method to prove the unified B-SPDE in \eq{bspdef} to be
well-posed in a suitable functional space. Although the approach is
partially embedded in the discussion of unique existence of solution
to a more general system with jumps in the preprint of
Dai~\cite{dai:newcla}. We refine it here and make it consistent with
the system in \eq{bspdef} to develop theoretical foundation of
random field based Malliavin calculus to conduct convergence
analysis and error estimation for our newly designed numerical
schemes.

Currently, there are numerous discussions concerning the numerical
schemes for resolving SDEs (see, e.g., Kloeden and
Platen~\cite{klopla:numsol}), SPDEs (see, e.g, Barth and
Lang~\cite{barlan:simsto}, Juan {\em et al.}~\cite{juaker:stomot}),
and BSDEs (see, e.g., Bender and Denk~\cite{benden:forsch}, Bouchard
and Touzi~\cite{boutou:disapp}, Bouchard and
ELIE~\cite{boueli:distim}, Gobet {\em et al.}~\cite{goblem:regmon},
Hu {\em et al.}~\cite{hunua:malcal}, Zhang~\cite{zha:numsch}).
Furthermore, there are also numerical techniques available in
computing the stationary distributions of reflecting SDEs (see,
e.g., Dai~\cite{dai:broapp}, Shen {\em et
al.}~\cite{shechedaidai:finele}). However, to the best of our
knowledge, there is no numerical technique available in the
literature for B-SPDEs. Thus, in this paper, we make such an attempt
to develop some numerical methods for the unified B-SPDE in
\eq{bspdef}.

More precisely, we design two algorithms to compute the adapted
solution of the B-SPDE in \eq{bspdef}. Comparing with most of the
existing schemes for BSDEs and considering computer implementation,
both of the algorithms are handled with completely discrete schemes
in terms of time and space. The first one (named
Algorithm~\ref{algorithmI}) is an iterative one while the second one
(named Algorithm~\ref{algorithmII}) is not a purely iterative one
since it needs to solve linear or nonlinear equations at each time
point. Hence, Algorithm~\ref{algorithmII} is expensive when $x$ is
in a higher-dimensional domain. Nevertheless, for the purpose of
comparison and for the case that $x$ is in a lower-dimensional
domain, Algorithm~\ref{algorithmII} is useful. Owing to the
similarity of discussions, the convergence analysis for the
algorithms is focused on Algorithm~\ref{algorithmI}.

The analysis concerning error estimation or rate of convergence of
Algorithm~\ref{algorithmI} is conducted with respect to a completely
discrete criterion. Comparing with existing discussions for BSDEs,
we need to develop new theory for random field based Malliavin
calculus to prove the existence and uniqueness of adapted solutions
to related first-order and second-order Malliavin derivative based
B-SPDEs under random environments.

The remainder of the paper is organized as follows. In
Section~\ref{uniexist}, we state conditions to guarantee the
existence and uniqueness of adapted solution to the B-SPDE in
\eq{bspdef}. In Section~\ref{finited}, we design our numerical
schemes and state our main convergence theorem. Related notations of
random field based calculus are also introduced. In
Section~\ref{uniextproof}, we prove our B-SPDE in \eq{bspdef} to be
well-posed. In Section~\ref{errorIthp}, we develop new theory for
random field based Malliavin calculus to provide theoretical
foundation in proving our main convergence theorem.

\section{The Adapted Solution to the B-SPDE: Existence and
Uniqueness}\label{uniexist}

\subsection{Preliminary Notations}

First, we use $(\Omega,{\cal F},\{{\cal F}_{t}\},P)$ to denote a
complete filtered probability space on which are defined a standard
$d$-dimensional Brownian motion $W\equiv\{W(t),t\in[0,T]\}$ with
$W(t)=(W_{1}(t),...,W_{d}(t))'$ and a filtration $\{{\cal
F}_{t},t\in[0,T]\}$ with ${\cal F}_{t}=\sigma(W(s),s\leq t)$, where
$T\in[0,\infty)$ and the prime denotes the corresponding transpose
of a matrix or a vector.

Second, we consider the $p$-dimensional rectangle
$D=[0,b_{1}]\times\cdot\cdot\cdot\times[0,b_{p}]$ with a given
$p\in{\cal N}=\{1,2,...\}$. Let $C^{k}(D,R^{q})$ for each
$k,q\in{\cal N}$ denote the Banach space of all functions $f$ having
continuous derivatives up to the order $k$ with the uniform norm,
\begin{eqnarray}
&&\|f\|_{C^{k}(D,q)}=\max_{c\in\{0,1,...,k\}}\max_{j\in\{1,...,r(c)\}}
\sup_{x\in D}\left|f^{(c)}_{j}(x)\right| \elabel{fckupd}
\end{eqnarray}
for each $f\in C^{k}(D,R^{q})$. The $r(c)$ in \eq{fckupd} for each
$c\in\{0,1,...,k\}$ is the total number of the partial derivatives
of the order $c$
\begin{eqnarray}
&&f^{(c)}_{r,(i_{1}...i_{p})}(x)=\frac{\partial^{c}f_{r}(x)}{\partial
x_{1}^{i_{1}}...\partial x_{p}^{i_{p}}} \elabel{difoperatorI}
\end{eqnarray}
with $i_{l}\in\{0,1,...,c\}$, $l\in\{1,...,p\}$, $r\in\{1,...,q\}$,
and $i_{1}+...+i_{p}=c$. Furthermore, let
\begin{eqnarray}
f_{(i_{1},...,i_{p})}^{(c)}&\equiv&(f_{1,(i_{1},...,i_{p})}^{(c)},...,
f_{q,(i_{1},...,i_{p})}^{(c)}),
\elabel{difoperatoro}\\
f^{(c)}(x)&\equiv&(f_{1}^{(c)}(x),...,f_{r(c)}^{(c)}(x)),
\elabel{difoperator}
\end{eqnarray}
where each $j\in\{1,...,r(c)\}$ corresponds to a $p$-tuple
$(i_{1},...,i_{p})$ and a $r\in\{1,...,q\}$.

Third, we use $C^{\infty}(D,R^{q})$ to denote the Banach space
\begin{eqnarray}
&&C^{\infty}(D,R^{q})\equiv
\left\{f\in\bigcap_{c=0}^{\infty}C^{c}(D,R^{q}),
\|f\|_{C^{\infty}(D,q)}<\infty\right\}, \elabel{cinfity}
\end{eqnarray}
where
\begin{eqnarray}
&&\|f\|^{2}_{C^{\infty}(D,q)}=\sum_{c=0}^{\infty}\xi(c)
\|f\|^{2}_{C^{c}(D,q)} \elabel{inftynorm}
\end{eqnarray}
for some discrete function $\xi(c)$ in terms of $c\in\{0,1,2,...\}$,
which is fast decaying in $c$. For convenience, we take
$\xi(c)=\frac{1}{((c^{10})!)(\eta(c)!)e^{c}}$ with
\begin{eqnarray}
&&\eta(c)=[\max\{|x_{1}|+...+|x_{p}|,x\in D\}]^{c}, \nonumber
\end{eqnarray}
where the notation $[]$ denotes the summation of the unity and the
integer part of a real number.

Fourth, we define some measurable spaces to support random fields
considered in this paper. Let $L^{2}_{{\cal
F}}([0,T],C^{\infty}(D;R^{q}))$ denote the set of all $R^{q}$-valued
(or called $C^{\infty}(D;R^{q})$-valued) measurable stochastic
processes $Z(t,x)$ adapted to $\{{\cal F}_{t},t\in[0,T]\}$ for each
$x\in D$, which are in $C^{\infty}(D,R^{q})$ for each fixed
$t\in[0,T]$), such that
\begin{eqnarray}
&&E\left[\int_{0}^{T}\|Z(t)\|^{2}_{C^{\infty}(D,q)}dt\right]<\infty.
\elabel{adaptednormI}
\end{eqnarray}
Let $L^{2}_{{\cal F},p}([0,T],C^{\infty}(D,R^{q}))$ denote the
corresponding set of predictable processes (see, e.g., Definition
5.2 and Definition 1.1 respectively in pages 21 and 45 of Ikeda and
Watanabe \cite{ikewat:stodif}). Furthermore, let $L^{2}_{{\cal
F}_{T}}(\Omega,C^{\infty}(D;R^{q}))$ denote the set of all
$R^{q}$-valued, ${\cal F}_{T}$-measurable random variables
$\zeta(x)$ for each $x\in D$, where $\zeta(x)\in
C^{\infty}(D,R^{q})$ satisfies
\begin{eqnarray}
&&\|\zeta\|^{2}_{L^{2}_{{\cal
F}_{T}}(\Omega,C^{\infty}(D,R^{q}))}\equiv
E\left[\|\zeta\|^{2}_{C^{\infty}(D,q)}\right]<\infty.
\elabel{initialx}
\end{eqnarray}
In addition, we define
\begin{eqnarray}
&&{\cal Q}^{2}_{{\cal F}}([0,T]\times D)\equiv L^{2}_{{\cal
F}}([0,T],C^{\infty}(D,R^{q}))\times L^{2}_{{\cal
F},p}([0,T],C^{\infty}(D,R^{q\times d})). \elabel{bunisoI}
\end{eqnarray}

\subsection{The Conditions}

In this subsection, we impose some conditions to guarantee the
unique existence of adapted solution to \eq{bspdef} and to be used
in the convergence analysis of our designed algorithms. First, let
``a.s." denote ``almost surely". Then, suppose that, for each
$s\in[0,T]$,
\begin{eqnarray}
\bar{V}(s,\cdot)&=&\left(\bar{V}_{1}(s,\cdot),...,\bar{V}_{d}(s,\cdot)
\right)\in C^{\infty}(D,R^{q\times d})\;\;\;a.s., \elabel{barvv}
\end{eqnarray}
and in \eq{bspdef}, ${\cal L}$ is a $q$-dimensional partial
differential operator satisfying the generalized Lipschitz condition
a.s.
\begin{eqnarray}
&&\left\|\Delta{\cal L}^{(c+l+o)}(s,x,u,v)\right\|\leq
K_{D,c}\left(\|u-v\|_{C^{k+c}(D,q)}+\|\bar{u}-\bar{v}\|_{C^{m+c}(D,qd)}
\right) \elabel{blipschitz}
\end{eqnarray}
for any $(u,\bar{u})$, $(v,\bar{v})$ $\in C^{\infty}(D,R^{q})\times
C^{\infty}(D,R^{q\times d})$, where $K_{D,c}$ with each
$c\in\{0,1,2,...\}$ is a nonnegative constant. Note that $K_{D,c}$
depends on the domain $D$ and the differential order $c$ with
respect to each $x\in D$ and may be unbounded as
$c\rightarrow\infty$ and $D\rightarrow R^{p}$. $l\in\{0,1,2\}$
denotes the $l$th order of partial derivative of $\Delta{\cal
L}^{(c)}(s,x,u,v)$ in time variable $t$. $o\in\{0,1,2\}$ denotes the
$o$th order of partial derivative of $\Delta{\cal
L}^{(c+l)}(s,x,u,v)$ in terms of a component of $u$ or $v$. $\|A\|$
is the largest absolute value of entries (or components) of the
given matrix (or vector) $A$, and
\begin{eqnarray}
&&\Delta{\cal L}^{(c+l+o)}(s,x,u,v) \equiv{\cal
L}^{(c+l+o)}(s,x,u,\bar{u})-{\cal L}^{(c+l+o)}(s,x,v,\bar{v}).
\nonumber
\end{eqnarray}
Similarly, ${\cal J}=({\cal J}_{1},...,{\cal J}_{d})$ is a $q\times
d$-dimensional partial differential operator satisfying, a.s.,
\begin{eqnarray}
&&\|\Delta{\cal J}^{(c+l+o)}(s,x,u,v)\|\leq
K_{D,c}\left(\|u-v\|_{C^{m+c}(D,q)}\right). \elabel{blipschitzo}
\end{eqnarray}
In addition, we assume that the generalized linear growth conditions
hold,
\begin{eqnarray}
\left\|{\cal L}^{(c+l+o)}(s,x,u,\bar{u})\right\|&\leq&
K_{D,c}\left(\delta_{0c}+\|u\|_{C^{k+c}(D,q)}+\|\bar{v}\|_{C^{m+c}(D,qd)}\right),
\elabel{blipschitzoI}\\
\left\|{\cal J}^{(c+l+o)}(s,x,u)\right\|&\leq&
K_{D,c}\left(\delta_{0c}+\|u\|_{C^{m+c}(D,q)}\right),
\elabel{blipschitzI}
\end{eqnarray}
where $\delta_{0c}=1$ if $c=0$ and $\delta_{0c}=0$ if $c>0$.

\subsection{The Adapted Solution}

\begin{theorem}\label{bsdeyI}
Assume that $H(x)\in L^{2}_{{\cal F}_{T}}(\Omega,
C^{\infty}(D;R^{q}))$ for each $x\in D$. Then, under conditions of
\eq{blipschitz}-\eq{blipschitzI}, if ${\cal L}(t,x,v,\cdot)$ and
${\cal J}(t,x,v,\cdot)$ are $\{{\cal F}_{t}\}$-adapted for each
fixed $x\in D$ and any given $(v,\bar{v})\in
C^{\infty}(D,R^{q})\times C^{\infty}(D,R^{q\times d})$ with
\begin{eqnarray}
&&{\cal L}(\cdot,x,0,\cdot)\in L^{2}_{{\cal
F}}\left([0,T],C^{\infty}(D,R^{q})\right),\;{\cal
J}(\cdot,x,0,\cdot)\in L^{2}_{{\cal
F}}\left([0,T],C^{\infty}(D,R^{q\times d})\right), \elabel{blipic}
\end{eqnarray}
the B-SPDE \eq{bspdef} has a unique adapted solution,
\begin{eqnarray}
&&(V(\cdot,\cdot),\bar{V}(\cdot,\cdot))\in {\cal Q}^{2}_{{\cal
F}}([0,T]\times D). \elabel{buniso}
\end{eqnarray}
\end{theorem}

The proof of Theorem~\ref{bsdeyI} is provided in
Section~\ref{uniextproof}, which is partially embedded in the
discussion of unique existence of solution to a more general system
with jumps in the preprint of Dai~\cite{dai:newcla}. Since the
techniques adopted in the proof are frequently used in the rest of
this paper, we refine them here and make them consistent with the
system in \eq{bspdef} for convenience.

\section{Numerical Schemes and Their Convergence}
\label{finited}

\subsection{The Schemes}

Consider a partition $\pi$ for the product of the time interval
$[0,T]$ and the $p$-dimensional rectangle
$D=[0,b_{1}]\times\cdot\cdot\cdot\times[0,b_{p}]$ with a given
$p\in{\cal N}=\{1,2,...\}$ as follows,
\begin{eqnarray}
\pi:
&&0=t_{0}<t_{1}<\cdot\cdot\cdot<t_{n_{0}}=T\;\;\;\;\mbox{with}\;\;
n_{0}\in\{0,1,...\},\elabel{partitionI}\\
&&0=x_{l}^{0}<x_{l}^{1}<\cdot\cdot\cdot<x_{l}^{n_{l}}=b_{l}\;\;
\mbox{with}\;\;l\in\{1,...,p\},\;n_{l}\in\{0,1,...\}.\nonumber
\end{eqnarray}
In the sequel, for all $l\in\{0,1,...,p\}$ and
$j_{l}\in\{1,...,n_{l}\}$, we take
\begin{eqnarray}
\Delta^{\pi}_{j_{0}}&=&t_{j_{0}}-t_{j_{0}-1},
\elabel{discI}\\
\Delta^{\pi}_{l}&=&x_{l}^{j_{l}}-x_{l}^{j_{l}-1}=\frac{b_{l}}{n_{l}},
\elabel{discII}\\
\Delta^{\pi}W_{j_{0}}&=&W(t_{j_{0}})-W(t_{j_{0}-1}),
\end{eqnarray}
and let
\begin{eqnarray}
|\pi|&\equiv&\max_{j_{0}\in\{1,...,n_{0}\},\;
l\in\{1,...,p\}}\left\{\Delta^{\pi}_{j_{0}},\;\Delta^{\pi}_{l}\right\},
\elabel{maxpartition}\\
D^{j_{1}...j_{p}}&\equiv&[x_{1}^{j_{1}-1},\;x_{1}^{j_{1}})\times\cdot
\cdot\cdot\times[x_{p}^{j_{p}-1},\;x_{p}^{j_{p}}),
\elabel{partitionD}\\
{\cal X}&\equiv&\left\{x:x=(x_{1}^{j_{1}},...,x_{p}^{j_{p}}),\;j_{l}
\in\{0,1,...,n_{l}\},\;l\in\{1,...,p\}\right\}. \elabel{xsetrepre}
\end{eqnarray}

To suitably describe the approximations of partial derivatives
appeared in \eq{bspdef} and \eq{difoperatorI}, we assume that the
orders $k$ and $m$ are less than $2\max\{n_{1},...,n_{p}\}$. Then we
can use the forward and the backward difference techniques to
approximate the partial derivatives in \eq{bspdef} and
\eq{difoperatorI} as follows. For each $f\in\{V,\bar{V}\}$,
$x\in{\cal X}$, $l\in\{1,...,p\}$, and each integer $c$ satisfying
$1\leq c\leq k$ or $m$ or $n$, define
\begin{eqnarray}
f^{(c)}_{i_{1}...(i_{l}+1)...i_{p},\pi}(t,x)
       &\equiv&\left\{\begin{array}{ll}
       \frac{f^{(c-1)}_{i_{1}...i_{p},\pi}(t,x+\Delta^{\pi}_{l}e_{l})
       -f^{(c-1)}_{i_{1}...i_{p},\pi}(t,x)}{\Delta^{\pi}_{l}}
       &\mbox{if}\;\;j_{l}<n_{l},\\
       \frac{f^{(c-1)}_{i_{1}...i_{p},\pi}(t,x-\Delta^{\pi}_{l}e_{l})
       -f^{(c-1)}_{i_{1}...i_{p},\pi}(t,x)}{\Delta^{\pi}_{l}}
       &\mbox{if}\;\;j_{l}=n_{l},
       \end{array}
\right. \elabel{fdifII}
\end{eqnarray}
where $e_{l}$ is the unit vector whose $l$th component is the unity
and others are zero, $f^{(0)}=f$, and $(i_{1},...,i_{p})\in{\cal
I}^{c-1}$ with
\begin{eqnarray}
&&{\cal I}^{c}\equiv\{(i_{1},...,i_{p}):
i_{1},...,i_{p}\;\;\mbox{are nonnegative integers
satisfying}\;\;i_{1}+...+i_{p}=c\}. \elabel{iindex}
\end{eqnarray}
Furthermore, we define the following vector for all given
$(i_{1},...,i_{p})\in{\cal I}^{c}$
\begin{eqnarray}
f^{(c)}_{\pi}(t,x) =(f^{(c)}_{i_{1}...i_{p},\pi}(t,x))
\elabel{vectorf}
\end{eqnarray}
according to an increasing order indexed by
$i_{p}c^{p}+i_{p-1}c^{p-1}+...+i_{2}c+i_{1}$. Next, to be simple for
notations, we define
\begin{eqnarray}
\;\;\;\;{\cal L}(t,x,V_{\pi}(t,x)) &\equiv&{\cal
L}(t,x,V_{\pi}(t,x),...,V^{(k)}_{\pi}(t,x),
\bar{V}_{\pi}(t,x),...,\bar{V}^{(m)}_{\pi}(t,x)),
\elabel{fdifIII}\\
{\cal J}(t,x,V_{\pi}(t,x)) &\equiv&{\cal
J}(t,x,V_{\pi}(t,x),...,V^{(n)}_{\pi}(t,x)) \elabel{fdifIV}
\end{eqnarray}
for each $x\in{\cal X}$. Thus, based on spacial discretization, we
can design the following direct discrete approximations of a
solution to the B-SPDE displayed in \eq{bspdef}.
\begin{algorithm}\label{algorithmI}
This algorithm is an iterative one in terms of
$\{V^{(c)}(t_{j_{0}},x),\bar{V}^{(c)}(t_{j_{0}},x)$ for all
$x\in{\cal X}\}$ with $j_{0}$ decreasing from $n_{0}$ to 1 in a
backward manner and $c=0,1,...,M$ with $M=\max\{m,n,k\}$,
\begin{eqnarray}
V^{(c)}_{i_{1}...i_{p},\pi}(t_{n_{0}},x)&=&H^{(c)}_{i_{1}...i_{p},\pi}(x),
\;\;
\bar{V}^{(c)}_{i_{1}...i_{p},\pi}(t_{n_{0}},x)=0,\elabel{IalgorithmI}\\
\;\;\;\;\;V^{(c)}_{i_{1}...i_{p},\pi}(t_{j_{0}-1},x)
&=&E\left[\left.V^{(c)}_{i_{1}...i_{p},\pi}(t_{j_{0}},x) +{\cal
L}^{(c)}_{i_{1}...i_{p},\pi}(t_{j_{0}},x,V_{\pi}(t_{j_{0}},x))
\Delta^{\pi}_{j_{0}}\right|{\cal F}_{t_{j_{0}-1}}\right],
\elabel{IIalgorithmI}\\
\;\;\;\;\;\bar{V}^{(c)}_{i_{1}...i_{p},\pi}(t_{j_{0}-1},x)
&=&\frac{1}{\Delta^{\pi}_{j_{0}}}
E\left[\left.V^{(c)}_{i_{1}...i_{p},\pi}(t_{j_{0}},x)\Delta^{\pi}W_{j_{0}}\right|{\cal
F}_{t_{j_{0}-1}}\right] \elabel{IIIalgorithmI}\\
&&+E\left[\left.{\cal
L}^{(c)}_{i_{1}...i_{p},\pi}(t_{j_{0}},x^{\pi},V_{\pi}(t_{j_{0}},x))
\Delta^{\pi}W_{j_{0}}\right|{\cal
F}_{t_{j_{0}-1}}\right]\nonumber\\
&&+{\cal
J}^{(c)}_{i_{1}...i_{p},\pi}(t_{j_{0}-1},x,V_{\pi}(t_{j_{0}-1},x)).
\nonumber
\end{eqnarray}
\end{algorithm}
\begin{algorithm}\label{algorithmII}
For all $x\in{\cal X}$ and $c=0,1,...,M$, we have the following
algorithm with respect to $j_{0}$ decreasing from $n_{0}$ to 1,
\begin{eqnarray}
V^{(c)}_{i_{1}...i_{p},\pi}(t_{n_{0}},x)
&=&H^{(c)}_{i_{1}...i_{p},\pi}(x)\elabel{IalgorithmII}\\
\;\;\;\;\;V^{(c)}_{i_{1}...i_{p},\pi}(t_{j_{0}-1},x)
&=&E\left[\left.V^{(c)}_{i_{1}...i_{p},\pi}(t_{j_{0}},x)\right|{\cal
F}_{t_{j_{0}-1}}\right]\elabel{IIalgorithmII}\\
&&+{\cal
L}^{(c)}_{i_{1}...i_{p},\pi}(t_{j_{0}-1},x,V_{\pi}(t_{j_{0}-1},x))
\Delta^{\pi}_{j_{0}}, \nonumber\\
\bar{V}^{(c)}_{i_{1}...i_{p},\pi}(t_{j_{0}-1},x)
&=&\frac{1}{\Delta^{\pi}_{j_{0}}}
E\left[\left.V^{(c)}_{i_{1}...i_{p},\pi}(t_{j_{0}},x)\Delta^{\pi}W_{j_{0}}\right|{\cal
F}_{t_{j_{0}-1}}\right]\elabel{IIIalgorithmII}\\
&&+{\cal
J}^{(c)}_{i_{1}...i_{p},\pi}(t_{j_{0}-1},x,V_{\pi}(t_{j_{0}-1},x)).
\nonumber
\end{eqnarray}
\end{algorithm}
In nature, Algorithm~\ref{algorithmII} is a sort of generalization
of the scheme considered in Bouchard and Touzi~\cite{boutou:disapp}
from BSDEs to the B-SPDEs. Comparing with
Algorithm~\ref{algorithmI}, it is not a purely iterative one since
it needs to solve linear or nonlinear equations at each time
$t_{j_{0}-1}$ to obtain $V^{(c)}_{i_{1}...i_{p},\pi}(t_{j_{0}-1},x)$
and $\bar{V}^{(c)}_{i_{1}...i_{p},\pi}(t_{j_{0}-1},x)$, which is
expensive when $x$ is in a higher-dimensional domain (e.g., $p\geq
3$). Furthermore, since the convergence analysis is similar, we will
focus our discussion on Algorithm~\ref{algorithmI} in the rest of
this paper.


\subsection{Additional Notations for Random Field Based Malliavin Calculus}

Let $H=L^{2}([0,T],R^{d})$ denote the separable Hilbert space of all
square integrable real-valued $d$-dimensional functions over the
time interval $[0,T]$ with inner product
\begin{eqnarray}
<\cdot,\cdot>_{H}=\int_{0}^{T}\langle h^{1}(t),h^{2}(t)\rangle dt
\;\;\mbox{for any}\;\;h^{1},h^{2}\in H, \nonumber
\end{eqnarray}
and
\begin{eqnarray}
\langle
h^{1}(t),h^{2}(t)\rangle=\sum_{i=1}^{d}h^{1}_{i}(t)h^{2}_{i}(t).
\elabel{hnorm}
\end{eqnarray}
For each $h\in H$, we define $W(h)=\int_{0}^{T}\langle
h(t),dW(t)\rangle$. Furthermore, let ${\cal S}$ denote the set of
all the random variables $F(x,\omega)$ of the following form with
$x\in D$ and $\omega\in\Omega$,
\begin{eqnarray}
F(x)=\phi(W(h^{1}),...,W(h^{g}),x)\;\;\mbox{with}\;\;\phi\in
C_{b}^{\infty}(R^{g+p},R^{q}),\;h^{1},...,h^{g}\in H,
\elabel{compoundf}
\end{eqnarray}
for some nonnegative integer $g$, where the lower index $b$ appeared
in $C_{b}^{\infty}$ means bounded. For each $F\in{\cal S}$, we
define
\begin{eqnarray}
&&\|F\|_{\alpha,2}^{\infty,2}
=\sum_{v=0}^{\infty}\xi(v)\left\|F\right\|^{v,2}_{\alpha,2},
\elabel{finfty}
\end{eqnarray}
where, the norm $\|\cdot\|^{v}_{\alpha,2}$ with $\alpha\in\{1,2\}$
and $v\in\{0,1,2,...\}$ is defined in the following way.

First, we define the first-order Malliavin derivative of the $c$th
order partial derivative $F^{(c)}_{r,i_{1}...i_{p}}(x)$ in terms of
$x\in D$ for each $r\in\{1,...,q\}$, $c\in\{0,1,...\}$, and
$(i_{1},...,i_{p})\in{\cal I}^{c}$ to be the $H$-valued random
variable,
\begin{eqnarray}
&&{\cal D}_{\theta_{1}}F^{(c)}_{r,i_{1}...i_{p}}(x)=\sum_{l=1}^{g}
\frac{\partial\phi^{(c)}_{r,i_{1}...i_{p}}}{\partial
y_{l}}(W(h^{1}),...,W(h^{g}),x)h^{l}(\theta_{1}),\;0\leq\theta_{1}\leq
T.\elabel{mmderi}
\end{eqnarray}
Second, for each $j\in\{1,...,d\}$, we define the associated
second-order Malliavin derivative
\begin{eqnarray}
&&{\cal D}_{\theta_{2}}{\cal
D}^{j}_{\theta_{1}}F^{(c)}_{r,i_{1}...i_{p}}(x)=\sum_{l=1}^{g}
\frac{\partial{\cal
D}^{j}_{\theta_{1}}F^{(c)}_{r,i_{1}...i_{p}}(x)}{\partial
y_{l}}h^{l}(\theta_{2}),\;0\leq\theta_{2}\leq T.\elabel{mmderiI}
\end{eqnarray}
Third, we define
\begin{eqnarray}
&&\left\|F\right\|^{v,2}_{1,2}
=E\left[\Lambda_{v}\left\|F^{(c)}_{r,i_{1}...i_{p}}\right\|^{2}_{C^{v}(D,1)}
+\int_{0}^{T}\Lambda_{v}\left\|{\cal
D}_{\theta_{1}}F^{(c)}_{r,i_{1}...i_{p}}\right\|^{2}_{C^{v}(D, d)}
d\theta_{1}\right],
\elabel{mnorm}\\
&&\left\|F\right\|^{v,2}_{2,2} =\left\|F\right\|^{v,2}_{1,2}
+E\left[\int_{0}^{T}\int_{0}^{T}\Lambda_{v}\left\|{\cal
D}_{\theta_{2}}{\cal
D}_{\theta_{1}}F^{(c)}_{r,i_{1}...i_{p}}\right\|^{2}_{C^{v}(D,d\times
d)}d\theta_{1}d\theta_{2}\right], \elabel{mnormI}
\end{eqnarray}
where the notation $\Lambda_{v}$ is defined by
\begin{eqnarray}
&&\Lambda_{v}=\max_{c\in\{0,1,...,v\}}
\max_{r\in\{1,...,q\},(i_{1},...,i_{p})\in{\cal
I}^{c}}.\elabel{notationI}
\end{eqnarray}
Note that if $F^{(c)}_{r,i_{1}...i_{p}}(x)$ for each $x\in D$ is
${\cal F}_{t}$-measurable, then ${\cal
D}_{\theta}F^{(c)}_{r,i_{1}...i_{p}}(x)=0$ for $\theta\in(t,T]$ and
we use ${\cal D}^{j}_{\theta}F^{(c)}_{r,i_{1}...i_{p}}(x)$ for each
$j\in\{1,..,d\}$ to denote the $j$th component of ${\cal
D}_{\theta}F^{(c)}_{r,i_{1}...i_{p}}(x)$.

Next, let $L^{2}(\Omega,C^{\infty}(D,R^{q}))$ be the space
corresponding to \eq{initialx} with no measurable property imposed,
and let $L^{2}_{\alpha,2}(\Omega,(C^{\infty}(D,H))^{q})$ be the
space of $H^{q}$($q$ product space $H\times...\times H$)-valued
processes, which is endowed with the norm \eq{finfty}. Then, we can
use ${\cal D}^{\alpha,2}_{\infty}$ to denote the domain of the
following unbounded operator,
\begin{eqnarray}
{\cal D}:\;L^{2}(\Omega,C^{\infty}(D,R^{q}))\rightarrow
L^{2}_{\alpha,2}(\Omega,(C^{\infty}(D,H))^{q}). \nonumber
\end{eqnarray}
Owing to Lemma~\ref{closablelemma} proved in
Section~\ref{errorIthp}, this domain is the closure of the class of
smooth random variables ${\cal S}$ with the norm \eq{finfty}. In the
sequel, we use $D_{\theta}F(x,\omega)$ with each $F\in{\cal
D}^{1,2}_{\infty}$ and $\theta\in[0,T]$ to denote the following
infinite-dimensional vector
\begin{eqnarray}
&&\left\{(D^{j}_{\theta}F^{(c)}_{r,i_{1}...i_{p}}(t,x):r\in\{1,...,q\},j\in\{1,...,d\},
c\in\{0,1,...\},(i_{1},...,i_{p})\in{\cal I}^{c}\right\}.
\elabel{dcomponent}
\end{eqnarray}
Similarly, we use $D_{\theta_{2}}D_{\theta_{1}}F(x,\omega)$ with
each $F\in{\cal D}^{2,2}_{\infty}$ and
$\theta_{1},\theta_{2}\in[0,T]$ to denote the following
infinite-dimensional vector
\begin{eqnarray}
&&\;\;\;\left\{(D^{j_{2}}_{\theta_{2}}D^{j_{1}}_{\theta_{1}}F^{(c)}_{r,i_{1}...i_{p}}(t,x):
r\in\{1,...,q\},j_{1},j_{2}\in\{1,...,d\},
c\in\{0,1,...\},(i_{1},...,i_{p})\in{\cal I}^{c}\right\}.
\elabel{dcomponent}
\end{eqnarray}

Finally, based on the above notations, we can impose the following
terminal conditions for the $q$-dimensional B-SPDEs in \eq{bspdef},
\begin{eqnarray}
H(x,\omega)&\in&{\cal D}^{2,2}_{\infty}\bigcap L^{2}_{{\cal
F}_{T}}(\Omega,C^{\infty}(D,R^{q})),
\elabel{termI}\\
D_{\theta_{1}}H(x,\omega)&\in&L^{2}_{{\cal
F}_{T}}(\Omega;C^{\infty}(D,R^{q\times
d})),\;\;\;\;\;\;\;\;\theta_{1}\in[0,T],
\elabel{termII}\\
{\cal D}_{\theta_{2}}{\cal
D}_{\theta_{1}}H(x,\omega)&\in&L^{2}_{{\cal
F}_{T}}(\Omega,C^{\infty}(D,R^{q\times d\times
d})),\;\;\;\;\theta_{2}\in[0,T]. \elabel{termIII}
\end{eqnarray}

\subsection{Convergence Theorem for Algorithm~\ref{algorithmI}}

\begin{theorem}\label{errorIth}
Consider Algorithm~\ref{algorithmI} under conditions required by
Theorem~\ref{bsdeyI} and with additional terminal conditions
\eq{termI}-\eq{termIII}. Then, there exists some nonnegative
constant $C$ depending only on the terminal time $T$, the region
$D$, and the constant $\kappa$ such that
\begin{eqnarray}
&&\sum_{c=0}^{M}\max_{x\in{\cal X}}\left(\sup_{t\in[0,T]}
E\left[\left\|\Delta V^{(c)}(t,x)\right\|^{2}\right]+
\sup_{t\in[0,T]}E\left[\left\|\Delta\bar{V}^{(c)}(t,x)\right\|^{2}\right]\right)\leq
C|\pi|, \elabel{errorest}
\end{eqnarray}
for all sufficiently small $|\pi|$, where
\begin{eqnarray}
\Delta V^{(c)}(t,x)&=&V^{(c)}(t,x)-V^{(c)}_{\pi}(t,x),\nonumber\\
\Delta\bar{V}^{(c)}(t,x)&=&\bar{V}^{(c)}(t,x)-\bar{V}^{(c)}_{\pi}(t,x),\nonumber\\
V^{(c)}_{\pi}(t,x)&=&V^{(c)}_{\pi}(t_{j_{0}-1},x),\;t\in[t_{j_{0}-1},t_{j_{0}}),
\;j_{0}\in\{n_{0},n_{0}-1,...,1\},
\nonumber\\
\bar{V}^{(c)}_{\pi}(t,x)&=&\bar{V}^{(c)}_{\pi}(t_{j_{0}-1},x),\;
t\in[t_{j_{0}-1},t_{j_{0}}).\nonumber
\end{eqnarray}
for each $c\in\{0,1,...,M\}$.
\end{theorem}

The proof of Theorem~\ref{errorIth} will be provided in
Section~\ref{errorIthp}.




\section{Proof of Theorem~\ref{bsdeyI}}\label{uniextproof}

We first prove three lemmas. Then, by combining these lemmas, we can
provide a proof for the theorem.

\subsection{Three Lemmas and Their Proofs}
\begin{lemma}\label{martindecom}
Under the conditions of Theorem~\ref{bsdeyI}, consider a tuplet for
each fixed $x\in D$,
\begin{eqnarray}
&&(U(\cdot,x),\bar{U}(\cdot,x))\in {\cal Q}^{2}_{{\cal
F}}([0,T]\times D).\elabel{threeu}
\end{eqnarray}
Then, there exists another tuplet $(V(\cdot,x),\bar{V}(\cdot,x))$
such that
\begin{eqnarray}
\;\;\;\;V(t,x)&=&H(x)+\int_{t}^{T}{\cal
L}(s,x,U,\cdot)ds+\int_{t}^{T}\left({\cal J}(s,x,U,\cdot)
-\bar{V}(s,x)\right)dW(s), \elabel{sigmanV}
\end{eqnarray}
where $V$ is a $\{{\cal F}_{t}\}$-adapted c\`adl\`ag process,
$\bar{V}$ is the corresponding predictable process. Furthermore, for
each $x\in D$,
\begin{eqnarray}
&&E\left[\int_{0}^{T}\|V(t,x)\|^{2}dt\right]<\infty,
\elabel{maxnormI}\\
&&E\left[\int_{0}^{T}\|\bar{V}(t,x)\|^{2}dt\right]<\infty.
\elabel{maxnormII}
\end{eqnarray}
\end{lemma}
\begin{proof}
For each fixed $x\in D$ and a tuplet $(U(\cdot,x),$
$\bar{U}(\cdot,x))$ as stated in \eq{threeu}, it follows from
conditions \eq{blipschitz}-\eq{blipic} that
\begin{eqnarray}
&&{\cal L}(\cdot,x,U,\cdot)\in L^{2}_{{\cal
F}}([0,T],C^{\infty}(D,R^{q})),\;{\cal J}(\cdot,x,U,\cdot)\in
L^{2}_{{\cal F}}([0,T],C^{\infty}(D,R^{q\times d})).
\elabel{ladaptedI}
\end{eqnarray}
Now, consider ${\cal L}$ and ${\cal J}$ in \eq{ladaptedI} as two new
starting ${\cal L}(\cdot,x,0,\cdot)$ and ${\cal
J}(\cdot,x,0,\cdot)$. Then, by the Martingale representation theorem
(see, e.g., Theorem 43 in page 186 of Protter~\cite{pro:stoint}), we
know that there is a unique predictable process $\bar{V}(\cdot,x)$
which is square-integrable for each $x\in D$ in the sense of
\eq{maxnormII} such that
\begin{eqnarray}
\;\;\;\;\hat{V}(t,x)&\equiv& \left.E\left[H(x) +\int_{0}^{T}{\cal
L}(s,x,U,\cdot)ds+\int_{0}^{T}{\cal J}(s,x,U,\cdot)dW(s)\right|{\cal
F}_{t}\right]
\elabel{sigmanI}\\
&=&\hat{V}(0,x)+\int_{0}^{t}\bar{V}(s,x)dW(s). \nonumber
\end{eqnarray}
Hence, we have,
\begin{eqnarray}
\;\;\;\hat{V}(0,x)&=&\hat{V}(T,x)-\int_{0}^{T}\bar{V}(s,x)dW(s)
\elabel{sigmanII}\\
&=&H(x)+\int_{0}^{T}{\cal L}(s,x,U,\cdot)ds +\int_{0}^{T}\left({\cal
J}(s,x,U,\cdot)-\bar{V}(s,x)\right)dW(s). \nonumber
\end{eqnarray}
Furthermore, owing to the Corollary in page 8 of
Protter~\cite{pro:stoint}, $\hat{V}(\cdot,x)$ can be taken as a
c\`adl\`ag process. Next, define a process $V$ given by
\begin{eqnarray}
V(t,x)&=& E\left[H(x) +\int_{t}^{T}{\cal
L}(s,x,U,\cdot)ds+\left.\int_{t}^{T}{\cal J}(s,x,U,\cdot)dW(s)
\right|{\cal F}_{t}\right]. \elabel{sigmanIII}
\end{eqnarray}
Then, it follows from \eq{blipschitzo}-\eq{blipschitzI} and simple
calculation that $V(\cdot,x)$ is square-integrable in the sense of
\eq{maxnormI}. Furthermore, by \eq{sigmanI}-\eq{sigmanIII}, we know
that
\begin{eqnarray}
V(t,x)&=&\hat{V}(t,x)-\int_{0}^{t}{\cal L}(s,x,U,\cdot)ds
-\int_{0}^{t}{\cal J}(s,x,U,\cdot)dW(s) \elabel{sigmanIV}
\end{eqnarray}
which indicates that $V(\cdot,x)$ is a c\`adl\`ag process. Now, for
a given tuplet $(U(\cdot,x),$ $\bar{U}(\cdot,x))$, it follows from
\eq{sigmanI}-\eq{sigmanII} and \eq{sigmanIV} that the corresponding
tuplet $(V(\cdot,x),$ $\bar{V}(\cdot,x))$ satisfies the equation
\eq{sigmanV} as stated in the lemma. Thus, we know that
\begin{eqnarray}
\;\;\;\;\;\;V(t,x)&\equiv&V(0,x)-\int_{0}^{t}{\cal L}(s,x,U,\cdot)ds
-\int_{0}^{t}\left({\cal J}(s,x,U,\cdot) -\bar{V}(s,x)\right)dW(s)
\elabel{sigmanVI}
\end{eqnarray}
Hence, we complete the proof of Lemma~\ref{martindecom}. $\Box$
\end{proof}
\begin{lemma}\label{differentiableV}
Under the conditions of Theorem~\ref{bsdeyI}, consider a tuplet as
in \eq{threeu} for each fixed $x\in D$ and define $V(t,x)$ and
$\bar{V}(t,x)$ by \eq{sigmanV}. Then, $(V^{(c)}(\cdot,x)$,
$\bar{V}^{(c)}(\cdot,x)$) for each $c\in\{0,1,...,\}$ exists a.s.
and satisfies
\begin{eqnarray}
V^{(c)}_{(i_{1}...i_{p})}(t,x)&=&H^{(c)}_{(i_{1}...i_{p})}(x)
+\int_{t}^{T}{\cal L}^{(c)}_{(i_{1}...i_{p})}(s,x,U,\cdot)ds
\elabel{pdsigmanV}\\
&&+\int_{t}^{T}\left({\cal J}^{(c)}_{(i_{1}...i_{p})}(s,x,U,\cdot)
-\bar{V}^{(c)}_{(i_{1}...i_{p})}(s,x)\right)dW(s), \nonumber
\end{eqnarray}
where $i_{1}+...+i_{p}=c$ and $i_{l}\in\{0,1,...,c\}$ with
$l\in\{1,...,p\}$. Furthermore,  $V^{(c)}_{(i_{1}...i_{p})}$ for
each $c\in\{0,1,...\}$ is a $\{{\cal F}_{t}\}$-adapted c\`adl\`ag
process, and $\bar{V}^{(c)}_{(i_{1}...i_{p})}$ is the corresponding
predictable processes. Both of them are square-integrable in the
senses of \eq{maxnormI}-\eq{maxnormII}.
\end{lemma}
\begin{proof}
First, we prove the claim in the lemma to be true for $c=1$. To do
so, for each given $t\in[0,T],x\in D$, and $(U(t,x),\bar{U}(t,x))$
as in the lemma, let
\begin{eqnarray}
&&(V^{(1)}_{(l)}(t,x),\bar{V}^{(1)}_{(l)}(t,x)) \elabel{tripletdv}
\end{eqnarray}
be defined by using \eq{sigmanV}, where ${\cal L}$ and ${\cal J}$
are replaced by their first-order partial derivatives ${\cal
L}^{(1)}_{(l)}$ and ${\cal J}^{(1)}_{(l)}$ in terms of $x_{l}$ with
$l\in\{1,...,p\}$. Then, we can prove that the tuplet defined in
\eq{tripletdv} for each $l$ is indeed the required first-order
partial derivative of $(V,\bar{V})$ that is defined by using
\eq{sigmanV} for the given $(U,\bar{U})$.

In fact, for each $f\in\{U,\bar{U},V,\bar{V}\}$, sufficiently small
positive constant $\delta$, and $l\in\{1,...,p\}$, define
\begin{eqnarray}
&&f_{(l),\delta}(t,x)\equiv f(t,x+\delta e_{l}),\elabel{fdeltan}
\end{eqnarray}
where $e_{l}$ is the unit vector whose $l$th component is one and
others are zero. Furthermore, let
\begin{eqnarray}
&&\Delta f^{(1)}_{(l),\delta}(t,x)
=\frac{f_{(l),\delta}(t,x)-f(t,x)}{\delta}-f^{(1)}_{(l)}(t,x)
\elabel{partialII}
\end{eqnarray}
for each $f\in\{U,\bar{U},V,\bar{V}\}$. In addition, let
\begin{eqnarray}
\Delta{\cal I}^{(1)}_{(l),\delta}(s,x,U)
&=&\frac{1}{\delta}\left({\cal I}(s,x+\delta e_{l},U(s,x+\delta
e_{l}),\cdot)-{\cal I}(s,x,U(s,x),\cdot)\right)
\elabel{deltasqr}\\
&&-{\cal I}^{(1)}_{(l)}(s,x,U(s,x),\cdot) \nonumber
\end{eqnarray}
for each ${\cal I}\in\{{\cal L},{\cal J}\}$, and let Tr$(A)$ denote
the trace of the matrix $A'A$ for a given matrix $A$. Then, by
applying \eq{sigmanVI} and the Ito's formula (see, e.g., Theorem 33
in page 81 of Protter~\cite{pro:stoint}) to the function
\begin{eqnarray}
&&\zeta(\Delta V^{(1)}_{(l),\delta}(t,x))\equiv\mbox{Tr}\left(\Delta
V_{(l),\delta}^{(1)}(t,x)\right)e^{2\gamma t}\nonumber
\end{eqnarray}
for some $\gamma>0$,  we see that
\begin{eqnarray}
&&\zeta(\Delta V^{(1)}_{(l),\delta}(t,x))+\int_{t}^{T}
\mbox{Tr}\left(\Delta{\cal J}^{(1)}_{(l),\delta}(s,x,U)
-\Delta\bar{V}^{(1)}_{(l),\delta}(s,x) \right)e^{2\gamma s}ds
\elabel{pvitod}\\
&=&2\int_{t}^{T}\left(-\gamma\mbox{Tr}\left(\Delta
V_{(l),\delta}^{(1)}(s,x)\right)+\left(\Delta
V_{(l),\delta}^{(1)}(s,x)\right)' \left(\Delta{\cal
L}^{(1)}_{(l),\delta}(s,x,U)\right)\right)e^{2\gamma
s}ds-M_{\delta}(t)
\nonumber\\
&\leq&\left(-2\gamma+\frac{3K^{2}_{D,1}}{\hat{\gamma}}\right)
\int_{t}^{T}\mbox{Tr}\left(\Delta
V_{(l),\delta}^{(1)}(s,x)\right)e^{2\gamma s}ds
+\hat{\gamma}\int_{t}^{T}\left\|\Delta{\cal
L}^{(1)}_{(l),\delta}(s,x,U)\right\|^{2}e^{2\gamma
s}ds-M_{\delta}(t)
\nonumber\\
&=&\hat{\gamma}\int_{t}^{T}\left\|\Delta{\cal
L}^{(1)}_{(l),\delta}(s,x,U)\right\|^{2}e^{2\gamma
s}ds-M_{\delta}(t) \nonumber
\end{eqnarray}
if, in the last equality, we take
\begin{eqnarray}
\hat{\gamma}=\frac{3K_{D,1}^{2}}{2\gamma}>0, \elabel{thgamma}
\end{eqnarray}
where $K_{D,1}$ is defined in \eq{blipschitz}-\eq{blipschitzI} and
$M_{\delta}(t)$ is a martingale given by
\begin{eqnarray}
2\sum_{j=1}^{d}\int_{t}^{T}\left(\Delta
V^{(1)}_{(l),\delta}(s,x)\right)'\left(\Delta({\cal
J}_{j})_{(l),\delta}^{(1)}(s,x,U)
-\Delta(\bar{V}_{j})_{(l),\delta}^{(1)}(s,x)\right)e^{2\gamma
s}dW_{j}(s). \nonumber
\end{eqnarray}

Next, by Lemma 1.3 in Peskir and Shiryaev~\cite{pesshi:optsto},
there is a sequence of $\{\delta_{n},n=1,2,...\}\subset[0,\sigma]$
for each $t\in[0,T]$ and $\sigma>0$ such that
\begin{eqnarray}
&&E\left[ess\sup_{0\leq\delta\leq\sigma}\zeta(\Delta
V^{(1)}_{(l),\delta}(t,x))\right] \elabel{difsuco}\\
&=&E\left[ess\sup_{\{\delta_{n}:0\leq\delta_{n}\leq\sigma,n=1,2,...\}}
\zeta(\Delta V^{(1)}_{(l),\delta_{n}}(t,x))\right]\nonumber\\
&=&\lim_{n\rightarrow\infty}E\left[\zeta(\Delta V^{(1)}_{(l),
\delta_{n}}(t,x))\right]\nonumber\\
&\leq& \hat{\gamma}\lim_{n\rightarrow\infty}
E\left[\int_{t}^{T}\left\|\Delta{\cal
L}^{(1)}_{(l),\delta_{n}}(s,x,U)\right\|^{2}e^{2\gamma
s}ds\right]-\lim_{n\rightarrow\infty}E\left[M_{\delta_{n}}(t)\right]
\nonumber\\
&\leq&
\hat{\gamma}E\left[\int_{t}^{T}ess\sup_{0\leq\delta\leq\sigma}
\left\|\Delta{\cal
L}^{(1)}_{(l),\delta}(s,x,U)\right\|^{2}e^{2\gamma s}ds\right],
\nonumber
\end{eqnarray}
where ``esssup'' denotes the essential supremum. Furthermore, the
first inequality in \eq{difsuco} is owing to \eq{pvitod}. Thus, by
the Lebesgue's dominated convergence theorem, we have
\begin{eqnarray}
&&\lim_{\sigma\rightarrow
0}E\left[ess\sup_{0\leq\delta\leq\sigma}\zeta(\Delta
V^{(1)}_{(l),\delta}(t,x))\right] \elabel{difsucI}\\
&\leq& \hat{\gamma}E\left[\int_{t}^{T}\lim_{\sigma\rightarrow
0}ess\sup_{0\leq\delta\leq\sigma} \left\|\Delta{\cal
L}^{(1)}_{(l),\delta}(s,x,U)\right\|^{2}e^{2\gamma s}ds\right],
\nonumber
\end{eqnarray}
where we have used the following fact owing to the mean-value
theorem and the conditions stated in \eq{blipschitzo},
\begin{eqnarray}
&&\left\|\Delta{\cal L}^{(1)}_{(l),\delta}(t,x,U)\right\| \leq
2K_{D,1}\left(\|U\|_{C^{k+1}(D,q)}+\|\bar{U}\|_{C^{m+1}(D,qd)}\right).
\nonumber
\end{eqnarray}
Thus, it follows from \eq{difsucI} and the Fatou's lemma that, for
any sequence $\sigma_{n}$ satisfying $\sigma_{n}\rightarrow 0$ along
$n\in{\cal N}$, there is a subsequence ${\cal N}'\subset{\cal N}$
such that
\begin{eqnarray}
&&ess\sup_{0\leq\delta\leq\sigma_{n}}\zeta(\Delta
V^{(1)}_{(l),\delta}(t,x))\rightarrow 0\;\;\mbox{along}\;\;n\in{\cal
N}'\;\;\mbox{a.s.}. \elabel{zetazero}
\end{eqnarray}
Therefore, we know that the first-order derivative of $V$ with
respect to $x_{l}$ for each $l\in\{1,...,p\}$ exists and equals
$V_{(l)}^{(1)}(t,x)$ a.s. for each $t\in[0,T]$ and $x\in D$.
Furthermore, it is $\{{\cal F}_{t}\}$-adapted. Now, by applying the
similar proof as used in \eq{difsuco}, we have
\begin{eqnarray}
&&\lim_{\sigma\rightarrow
0}E\left[\int_{t}^{T}ess\sup_{0\leq\delta\leq\sigma}
\mbox{Tr}\left(\Delta{\cal J}^{(1)}_{(l),\delta}(s,x,U)
-\Delta\bar{V}^{(1)}_{(l),\delta}(s,x) \right)e^{2\gamma s}ds\right]
\elabel{difsucII}\\
&\leq& \hat{\gamma}E\left[\int_{t}^{T}\lim_{\sigma\rightarrow
0}ess\sup_{0\leq\delta\leq\sigma} \left\|\Delta{\cal
L}^{(1)}_{(l),\delta}(s,x,U)\right\|^{2}e^{2\gamma s}ds\right].
\nonumber
\end{eqnarray}
Hence, it follows from \eq{zetazero} and \eq{difsucII} that
\begin{eqnarray}
&&\lim_{\delta\rightarrow
0}\Delta\bar{V}^{(1)}_{(l),\delta}(t,x)=\lim_{\delta\rightarrow
0}\Delta{\cal J}^{(1)}_{(l),\delta}(t,x,U)=0.\;\;\;\mbox{a.s.}
\nonumber
\end{eqnarray}
Thus, we know that the first-order derivative of $\bar{V}$ in terms
of $x_{l}$ for each $l\in\{1,...,p\}$ exists and equals
$\bar{V}^{(1)}_{(l)}(t,x)$ a.s. for every $t\in[0,T]$ and $x\in D$.
Furthermore, it is a $\{{\cal F}_{t}\}$-predictable process.

Second, supposing that $(V^{(c-1)}(t,x),$ $\bar{V}^{(c-1)}(t,x))$
associated with a given $(U(t,x),$ $\bar{U}(t,x))$ $\in{\cal
Q}^{2}_{{\cal F}}([0,T])$ exists for any given $c\in\{1,2,...\}$.
Then, we can prove that
\begin{eqnarray}
&&\left(V^{(c)}(t,x), \bar{V}^{(c)}(t,x)\right)
\elabel{cthderivative}
\end{eqnarray}
exists for the given $c$. In doing so, for any fixed nonnegative
integer numbers $i_{1},...,i_{p}$ satisfying $i_{1}+...+i_{p}=c-1$
for the given $c\in\{1,2,...\}$, any $f\in\{V,\bar{V}\}$, any
$l\in\{1,...,p\}$, and any small enough $\delta>0$,  we define
\begin{eqnarray}
&&f_{(i_{1}...(i_{l}+1)...i_{p}),\delta}^{(c-1)}(t,x)\equiv
f_{(i_{1}...i_{p})}^{(c-1)}(t,x+\delta e_{l}),\elabel{fdeltan}
\end{eqnarray}
which corresponds to ${\cal I}^{(c-1)}_{(i_{1}...i_{p})}(s,x+\delta
e_{l},U(s,x+\delta e_{l}),\cdot)$ with ${\cal I}\in\{{\cal L},{\cal
J}\}$ via \eq{sigmanV}, where the differential operators ${\cal L}$
and ${\cal J}$ are replaced by their $(c-1)$th-order partial
derivatives ${\cal L}^{(c-1)}_{(i_{1}...i_{p})}$ and ${\cal
J}^{(c-1)}_{(i_{1}...i_{p})}$. Similarly, let
\begin{eqnarray}
(V^{(c)}_{(i_{1}...(i_{l}+1)...i_{p})}(t,x),\;
\bar{V}^{(c)}_{(i_{1}...(i_{l}+1)...i_{p})}(t,x)) \nonumber
\end{eqnarray}
be defined by using \eq{sigmanV}, where ${\cal L}$ and ${\cal J}$
are replaced by their $c$th-order partial derivatives ${\cal
L}^{(c)}_{i_{1}...(i_{l}+1)...i_{p}}$ and ${\cal
J}^{(c)}_{i_{1}...(i_{l}+1)...i_{p}}$ corresponding to a given $t,x,
U(t,x),$ $\bar{U}(t,x)$. Furthermore, set
\begin{eqnarray}
&&\Delta f^{(c)}_{(i_{1}...(i_{l}+1)...i_{p}),\delta}(t,x)
=\frac{f^{(c-1)}_{(i_{1}...(i_{l}+1)...i_{p}),\delta}(t,x)
-f^{(c-1)}_{(i_{1}...i_{p})}(t,x)}{\delta}
-f^{(c)}_{(i_{1}...(i_{l}+1)...i_{p})}(t,x) \elabel{ipartialII}
\end{eqnarray}
for each $f\in\{U,\bar{U},V,\bar{V}\}$, and let
\begin{eqnarray}
&&\Delta{\cal I}^{(c)}_{(i_{1}...(i_{l}+1)...i_{p}),\delta}(t,x,U)
\elabel{ideltasqr}\\
&=&\frac{1}{\delta}\left({\cal
I}^{(c-1)}_{(i_{1}...i_{p})}(t,x+\delta e_{l},U(t,x+\delta
e_{l}),\cdot)-{\cal
I}^{(c-1)}_{(i_{1}...i_{p})}(s,x,U(s,x),\cdot)\right)\nonumber\\
&&-{\cal I}^{(c)}_{(i_{1}...(i_{l}+1)...i_{p})}(s,x,U(s,x),\cdot)
\nonumber
\end{eqnarray}
for ${\cal I}\in\{{\cal L},{\cal J}\}$. Then, by the It$\hat{o}$'s
formula and repeating the procedure as used in the second step, we
know that
\begin{eqnarray}
(V^{(c)}_{(i_{1}...(i_{l}+1)...i_{p})}(t,x),\;
\bar{V}^{(c)}_{(i_{1}...(i_{l}+1)...i_{p})}(t,x)) \nonumber
\end{eqnarray}
exist for the given $c\in\{1,2,...\}$ and all $l\in\{1,...,p\}$.
Thus, we know that the claim in \eq{cthderivative} is correct.

Third, by the induction method in terms of $c\in\{1,2,...\}$, we
know that the claims stated in the lemma are right. Therefore, we
complete the proof of Lemma~\ref{differentiableV}. $\Box$
\end{proof}

To state our next lemma, we let $D_{{\cal
F}}^{2}([0,T],C^{\infty}(D,R^{q}))$ be the set of $R^{q}$-valued
$\{{\cal F}_{t}\}$-adapted and square integrable c\`adl\`ag
processes as in \eq{adaptednormI}. Furthermore, for any given number
sequence $\gamma=\{\gamma_{c},c=0,1,2,...\}$ with $\gamma_{c}\in R$,
define ${\cal M}^{D}_{\gamma}[0,T]$ to be the following Banach space
(see, e.g., the similar explanation as used in Yong and
Zhou~\cite{yonzho:stocon}, and Situ~\cite{sit:solbac})
\begin{eqnarray}
{\cal M}^{D}_{\gamma}[0,T]&=&D^{2}_{{\cal
F}}([0,T],C^{\infty}(D,R^{q}))\times L^{2}_{{\cal
F},p}([0,T],C^{\infty}(D,R^{q\times d}))
\elabel{combanach}
\end{eqnarray}
endowed with the norm: for any given $(U,\bar{U})\in{\cal
M}^{D}_{\gamma}[0,T]$,
\begin{eqnarray}
\;\;\;\left\|(U,\bar{U})\right\|^{2}_{{\cal M}^{D}_{\gamma}}
&\equiv& \sum_{c=0}^{\infty}\xi(c)\left\|(U,\bar{U})
\right\|^{2}_{{\cal M}^{D}_{\gamma_{c},c}}, \elabel{comnorm}
\end{eqnarray}
where, without loss of generality, we assume that $m=k$ in
\eq{bspdef} and
\begin{eqnarray}
\;\;\;\;\left\|(U,\bar{U})\right\|^{2}_{{\cal
M}^{D}_{\gamma_{c},c}}&=&E\left[\sup_{0\leq t\leq
T}\left\|U(t)\right\|^{2}_{C^{c}(D,q)}e^{2\gamma_{c}t}\right]
+E\left[\int_{0}^{T}
\left\|\bar{U}(t)\right\|^{2}_{C^{c}(D,qd)}e^{2\gamma_{c}t}dt\right].
\elabel{kcomnorm}
\end{eqnarray}
Then, we have the following lemma.
\begin{lemma}\label{lemmathree}
Under the conditions of Theorem~\ref{bsdeyI}, all the claims in the
theorem are true.
\end{lemma}
\begin{proof}
First, by using \eq{sigmanV}, we can define the following map,
\begin{eqnarray}
&&\Xi:(U(\cdot,x),\bar{U}(\cdot,x))
\rightarrow(V(\cdot,x),\bar{V}(\cdot,x)).\nonumber
\end{eqnarray}
Then, based on the norm defined in \eq{comnorm}, we can show that
$\Xi$ forms a contraction mapping in ${\cal M}^{D}_{\gamma}[0,T]$.
In fact, for $i\in\{1,2,...\}$, consider the following sequence of
processes,
\begin{eqnarray}
&&(U^{i}(\cdot,x),\bar{U}^{i}(\cdot,x))\in{\cal
M}^{D}_{\gamma}[0,T],
\nonumber\\
&&(U^{i+1}(\cdot,x),\bar{U}^{i+1}(\cdot,x))=\Xi(U^{i}(\cdot,x),
\bar{U}^{i}(\cdot,x)). \nonumber
\end{eqnarray}
Furthermore, define
\begin{eqnarray}
&&\Delta f^{i}=f^{i+1}-f^{i}\;\;\;\mbox{with}\;\;\;
f\in\left\{U,\bar{U}\right\},\nonumber
\end{eqnarray}
and take
\begin{eqnarray}
&&\zeta(\Delta U^{i}(t,x))=\mbox{Tr}\left(\Delta
U^{i}(t,x)\right)e^{2\gamma_{0}t}. \elabel{firstzeta}
\end{eqnarray}
Then, by using \eq{blipschitz} and the similar argument as used in
proving \eq{pvitod}, we know that, for a $\gamma_{0}>0$ and each
$i\in\{2,3,...\}$,
\begin{eqnarray}
&&\zeta(\Delta U^{i}(t,x))+\int_{t}^{T}\mbox{Tr}\left(\Delta{\cal
J}(s,x,U^{i},U^{i-1}) -\Delta\bar{U}^{i}(s,x)
\right)e^{2\gamma_{0}s}ds
\elabel{eupvitod}\\
&&\leq\hat{\gamma}_{0}\int_{t}^{T}\left\|\Delta{\cal
L}(s,x,U^{i},U^{i-1})\right\|^{2}e^{2\gamma_{0}s}ds-M^{i}(t)
\nonumber\\
&&\leq\hat{\gamma}_{0}K_{a,0}N^{i-1}(t)-M^{i}(t),
\nonumber
\end{eqnarray}
where $K_{a,0}$ is some nonnegative constant depending only on
$K_{D,0}$. For the last inequality in \eq{eupvitod}, we have taken
\begin{eqnarray}
&&\hat{\gamma}_{0}=\frac{3K_{D,0}^{2}}{2\gamma_{0}}>0.
\elabel{euthgamma}
\end{eqnarray}
Furthermore, $N^{i-1}(t)$ appeared in \eq{eupvitod} is given by
\begin{eqnarray}
&&\;N^{i-1}(t)=\int_{t}^{T}\left(\left\|\Delta
U^{i-1}(s)\right\|^{2}_{C^{k}(D,q)}+\left\|\Delta\bar{U}^{i-1}(s)
\right\|^{2}_{C^{k}(D,qd)}\right)e^{2\gamma_{0}s}ds
\elabel{euntt}
\end{eqnarray}
and $M^{i}(t)$ is a martingale of the following form,
\begin{eqnarray}
&&M^{i}(t)=-2\sum_{j=1}^{d}\int_{t}^{T}\left((\Delta
U^{i})(s,x)\right)'\left(\Delta{\cal J}_{j}(s,x,U^{i},U^{i-1})
-(\Delta\bar{U}^{i})_{j}(s,x)\right) e^{2\gamma_{0}s}dW_{j}(s).
\elabel{eumsigman}
\end{eqnarray}
Then, by applying \eq{eupvitod}-\eq{eumsigman} and the martingale
properties related to stochastic integral, we have
\begin{eqnarray}
&&E\left[\left\|\Delta U^{i}(t,x) \right\|^{2}e^{2\gamma_{0}
t}+\int_{t}^{T}\mbox{Tr}\left(\Delta{\cal J}(s,x,U^{i},U^{i-1})
-\Delta\bar{U}^{i}(s,x) \right)e^{2\gamma_{0} s}ds\right.
\elabel{vitodI}\\
&&\leq\hat{\gamma}_{0}(T+1)K_{a,0}\left\|(\Delta
U^{i-1},\Delta\bar{U}^{i-1}) \right\|^{2}_{{\cal
M}^{D}_{\gamma_{0},k}}. \nonumber
\end{eqnarray}
Thus, by using \eq{eumsigman} and the Burkholder-Davis-Gundy's
inequality (see, e.g., Theorem 48 in page 193 of
Protter~\cite{pro:stoint}), we have,
\begin{eqnarray}
&&\;\;\;E\left[\sup_{0\leq t\leq T}\left|M^{i}(t)\right|\right]
\elabel{vitodII}\\
&\leq&4\sum_{j=1}^{d}E\left[\sup_{0\leq t\leq
T}\left|\int_{0}^{t}\left(\Delta U^{i}(s,x)\right)'\left(\Delta{\cal
J}_{j}(s,x,U^{i},U^{i-1})-(\Delta\bar{U}^{i})_{j}
(s,x)\right)e^{2\gamma_{0} s}dW_{j}(s)\right|\right]
\nonumber\\
&\leq&K_{b,0}\sum_{j=1}^{d}E\left[\left(\int_{0}^{T}\left\|\Delta
U^{i}(s,x)\right\|^{2}\left\|(\Delta{\cal
J}^{i})_{j}(s,x,U^{i},U^{i-1})-(\Delta\bar{U}^{i})_{j}
(s,x)\right\|^{2}e^{4\gamma_{0} s}ds\right)^{\frac{1}{2}}
\right]\nonumber\\
&\leq&K_{b,0}E\left[\left(\sup_{0\leq t\leq T}\|\Delta
U^{i}(t,x)\|^{2}e^{2\gamma_{0} t}\right)^{\frac{1}{2}}\right.
\nonumber\\
&&\;\;\;\;\;\;\;\;\;\;\left(\sum_{j=1}^{d}\left(\int_{0}^{T}
\left\|\Delta{\cal
J}_{j}(s,x,U^{i},U^{i-1})-(\Delta\bar{U}^{i})_{j}(s,x)
\right\|^{2}e^{2\gamma_{0} s}ds\right)^{\frac{1}{2}}\right]
\nonumber\\
&\leq&\frac{1}{2}E\left[\sup_{0\leq t\leq T}\|\Delta
U^{i}(t,x)\|^{2}e^{2\gamma_{0} t}\right]\nonumber\\
&&+dK_{b,0}^{2}E\left[\left(\int_{0}^{T}\mbox{Tr}\left(\Delta{\cal
J}(s,x,U^{i},U^{i-1})-(\Delta\bar{U}^{i})(s,x)
\right)e^{2\gamma_{0} s}ds\right)\right]\nonumber\\
&\leq&\frac{1}{2}E\left[\sup_{0\leq t\leq T}\|\Delta
U^{i}(t,x)\|^{2}_{C^{0}(q)}e^{2\gamma_{0}
t}\right]+\hat{\gamma}_{0}(T+1)dK_{a,0}K_{b,0}^{2}\left\|(\Delta
U^{i-1},\Delta\bar{U}^{i-1}) \right\|^{2}_{{\cal
M}^{D}_{\gamma_{0},k}}, \nonumber
\end{eqnarray}
where $K_{b,0}$ is some nonnegative constant depending only on
$K_{D,0}$ and $T$. The last inequality of \eq{vitodII} is owing to
\eq{vitodI}. Therefore, by using \eq{eupvitod}-\eq{vitodII}, we know
that
\begin{eqnarray}
&&E\left[\sup_{0\leq t\leq T}\left\|\Delta
U^{i}(t)\right\|^{2}_{C^{0}(q)}e^{2\gamma_{0}
t}\right]\elabel{firstineu}\\
&&\leq
2\left(1+dK_{b,0}^{2}\right)K_{a,0}\hat{\gamma}_{0}(T+1)\left\|(\Delta
U^{i-1},\Delta\bar{U}^{i-1}) \right\|^{2}_{{\cal
M}^{D}_{\gamma_{0},k}}. \nonumber
\end{eqnarray}
Furthermore, by using \eq{eupvitod} and \eq{blipschitzo}, we know
that, for each $i\in\{3,4,...\}$,
\begin{eqnarray}
&&E\left[\int_{t}^{T}
\mbox{Tr}\left(\Delta\bar{U}^{i}(s,x)\right)e^{2\gamma_{0}
s}ds\right]
\elabel{secondineu}\\
&\leq& 2E\left[\int_{t}^{T}\mbox{Tr}\left(\Delta{\cal
J}(s,x,U^{i},U^{i-1}) -\Delta\bar{U}^{i}(s,x)
\right)e^{2\gamma_{0} s}ds\right]\nonumber\\
&&+2E\left[\int_{t}^{T}\mbox{Tr}\left(\Delta{\cal
J}(s,x,U^{i},U^{i-1})\right)e^{2\gamma_{0} s}ds\right]
\nonumber\\
&\leq&2\hat{\gamma}_{0}K_{C,0}\left(\left\|(\Delta
U^{i-1},\Delta\bar{U}^{i-1}) \right\|^{2}_{{\cal
M}^{D}_{\gamma_{0},k}}+\left\|(\Delta U^{i-2},\Delta\bar{U}^{i-2})
\right\|^{2}_{{\cal M}^{D}_{\gamma_{0},k}}\right), \nonumber
\end{eqnarray}
where $K_{C,0}$ is some nonnegative constant depending only on
$K_{D,0}$ and $T$. Hence, by \eq{eupvitod},
\eq{firstineu}-\eq{secondineu}, and the fact that all functions and
norms used in this paper are continuous in terms of $x$, we have
that
\begin{eqnarray}
&&\left\|(\Delta U^{i},\Delta\bar{U}^{i}) \right\|^{2}_{{\cal
M}^{D}_{\gamma_{0},0}}
\elabel{lastine}\\
&&\leq\hat{\gamma}_{0}K_{d,0}\left(\left\|(\Delta
U^{i-1},\Delta\bar{U}^{i-1}) \right\|^{2}_{{\cal
M}^{D}_{\gamma_{0},k}}+\left\|(\Delta U^{i-2},\Delta\bar{U}^{i-2})
\right\|^{2}_{{\cal M}^{D}_{\gamma_{0},k}}\right), \nonumber
\end{eqnarray}
where $K_{d,0}$ is some nonnegative constant depending only on
$K_{D,0}$ and $T$.

Next, by using Lemma~\ref{differentiableV} and the similar
construction as used in \eq{firstzeta}, we can define
\begin{eqnarray}
&&\zeta(\Delta U^{c,i}(t,x))\equiv\mbox{Tr}\left(\Delta
U^{c,i}(t,x)\right)e^{2\gamma_{c} t} \elabel{secondzeta}
\end{eqnarray}
for each $c\in\{1,2,...\}$, where
\begin{eqnarray}
&&\Delta U^{c,i}(t,x))=(\Delta U^{(0),i}(t,x)),\Delta
U^{(1),i}(t,x)),...,\Delta U^{(c),i}(t,x))'. \nonumber
\end{eqnarray}
Thus, by using the It$\hat{o}$'s formula and the similar discussion
for \eq{lastine}, we have that
\begin{eqnarray}
&&\left\|(\Delta U^{i},\Delta\bar{U}^{i}) \right\|^{2}_{{\cal
M}^{D}_{\gamma_{c},c}}
\elabel{clastine}\\
&\leq&\hat{\gamma}_{c}K_{d,c}\left(\left\|(\Delta
U^{i-1},\Delta\bar{U}^{i-1}) \right\|^{2}_{{\cal
M}^{D}_{\gamma_{c},k+c}}+\left\|(\Delta U^{i-2},\Delta\bar{U}^{i-2})
\right\|^{2}_{{\cal M}^{D}_{\gamma_{c},k+c}}\right)
\nonumber\\
&\leq&\frac{\delta}{((c+1)^{10}(c+2)^{10}...(c+k)^{10})(\eta(c+1)
\eta(c+2)...\eta(c+k))}
\nonumber\\
&&\left(\left\|(\Delta U^{i-1},\Delta\bar{U}^{i-1})
\right\|^{2}_{{\cal M}^{D}_{\gamma_{k+c},k+c}}+\left\|(\Delta
U^{i-2},\Delta\bar{U}^{i-2}) \right\|^{2}_{{\cal
M}^{D}_{\gamma_{k+c},k+c}}\right),\nonumber
\end{eqnarray}
where, for the last inequality of \eq{clastine}, we have taken the
number sequence $\gamma$ such that $\gamma_{0}<\gamma_{1}<...$ and
\begin{eqnarray}
\hat{\gamma}_{c}K_{d,c}((c+1)^{10}(c+2)^{10}...(c+k)^{10})(\eta(c+1)
\eta(c+2)...\eta(c+k))\leq\delta\nonumber
\end{eqnarray}
for some $\delta>0$ such that $2\sqrt{e^{k}\delta}$ is sufficiently
small. Hence, we know that
\begin{eqnarray}
&&\left\|(\Delta U^{i},\Delta\bar{U}^{i}) \right\|^{2}_{{\cal
M}^{D}_{\gamma}}\leq e^{k}\delta\left(\left\|(\Delta
U^{i-1},\Delta\bar{U}^{i-1}) \right\|^{2}_{{\cal
M}^{D}_{\gamma}}+\left\|(\Delta U^{i-2},\Delta\bar{U}^{i-2})
\right\|^{2}_{{\cal M}^{D}_{\gamma}}\right).
\elabel{infinityine}
\end{eqnarray}
Owing to $(a^{2}+b^{2})^{1/2}\leq a+b$ for $a,b\geq 0$, we can
conclude that
\begin{eqnarray}
&&\left\|(\Delta U^{i},\Delta\bar{U}^{i}) \right\|_{{\cal
M}^{D}_{\gamma}}\leq\sqrt{e^{k}\delta}\left(\left\|(\Delta
U^{i-1},\Delta\bar{U}^{i-1})\right\|_{{\cal
M}^{D}_{\gamma}}+\left\|(\Delta
U^{i-2},\Delta\bar{U}^{i-2})\right\|_{{\cal M}^{D}_{\gamma}}\right).
\elabel{noninfinityine}
\end{eqnarray}
Thus,
it follows from \eq{noninfinityine} that
\begin{eqnarray}
\;\;\;\;\;\;\;\sum_{i=3}^{\infty}\left\|(\Delta
U^{i},\Delta\bar{U}^{i})\right\|_{{\cal M}^{D}_{\gamma}}
&\leq&\frac{\sqrt{e^{k}\delta}}{1-2\sqrt{e^{k}\delta}}\left(2\left\|(\Delta
U^{2},\Delta\bar{U}^{2})\right\|_{{\cal
M}^{D}_{\gamma}}+\left\|(\Delta
U^{1},\Delta\bar{U}^{1})\right\|_{{\cal
M}^{D}_{\gamma}}\right) \elabel{seriescon}\\
&<&\infty.\nonumber
\end{eqnarray}
Therefore, by using \eq{seriescon}, we see that
$(U^{i},\bar{U}^{i})$ with $i\in\{1,2,...\}$ forms a Cauchy sequence
in ${\cal M}^{D}_{\gamma}[0,T]$. Hence, there is some $(U,\bar{U})$
such that
\begin{eqnarray}
&&(U^{i},\bar{U}^{i})\rightarrow
(U,\bar{U})\;\;\mbox{as}\;\;i\rightarrow\infty\;\;
\mbox{in}\;\;{\cal M}^{D}_{\gamma}[0,T]. \elabel{finalcon}
\end{eqnarray}
Finally, by using \eq{finalcon} and the similar procedure as used
for Theorem 5.2.1 in pages 68-71 of
$\emptyset$ksendal~\cite{oks:stodif}, the proof of
Lemma~\ref{lemmathree} is completed. $\Box$
\end{proof}

\subsection{Proof of Theorem~\ref{bsdeyI}}

By combining Lemma~\ref{martindecom}-Lemma~\ref{lemmathree}, we can
reach a proof for Theorem~\ref{bsdeyI}. $\Box$

\section{Proof of Theorem~\ref{errorIth}}\label{errorIthp}

To prove the theorem, we first develop new fundamental theory for
random field based Malliavin Calculus in
subsections~\ref{basicp}-\ref{prioriI}. Then, based on this newly
developed theory, we provide a proof for Theorem~\ref{errorIth} in
subsections~\ref{repres}-\ref{proofcon}.

\subsection{Basic Properties of Random Field Based Malliavin
Calculus}\label{basicp}

\begin{lemma}\label{closablelemma}
The unbounded operator defined in \eq{mmderi} is closable from
$L^{2}(\Omega,C^{\infty}(D,R^{q}))$ to
$L^{2}_{\alpha,2}(\Omega,(C^{\infty}(D,H))^{q})$ with
$\alpha\in\{1,2\}$.
\end{lemma}
\begin{proof}
First, we consider the case that $\alpha=1$. Let
$\{F^{i}:i\in\{1,2,...\}\}$ be a sequence of smooth random
variables, which converges to zero along $i\in\{1,2,...\}$ in
$L^{2}(\Omega,C^{\infty}(D,R^{q}))$. Thus, we can conclude that
$F^{(c),i}_{r,i_{1}...i_{p}}(x)\rightarrow 0$ along
$i\in\{1,2,...\}$ in the usual mean-square sense for each $x\in D$,
$r\in\{1,...,q\}$, $c\in\{0,1,...\}$, and $(i_{1},...,i_{p})\in{\cal
I}^{c}$. In the meanwhile, we suppose that the corresponding
sequence related to Malliavin derivatives converges to some $\eta$
in $L^{2}_{1,2}(\Omega,(C^{\infty}(D,H))^{q})$, which implies that
$DF^{(c),i}_{r,i_{1}...i_{p}}(x)\rightarrow\eta^{(c)}_{r,i_{1}...i_{p}}(x)$
along $i\in\{1,2,...\}$ in the usual mean-square sense for each
$x\in D$, $r\in\{1,...,q\}$, $c\in\{0,1,...\}$, and
$(i_{1},...,i_{p})\in{\cal I}^{c}$. Then, it follows from the proof
of Proposition 1.2.1 in page 26 of Nualart~\cite{nua:malcal} that
$\eta^{(c)}_{r,i_{1}...i_{p}}(x)=0$ for each $x\in D$,
$r\in\{1,...,q\}$, $c\in\{0,1,...\}$, and $(i_{1},...,i_{p})\in{\cal
I}^{c}$. Thus, we know that $\eta=0$. Hence, by the definition of
the closable operator (see, e.g., page 77 of
Yosida~\cite{yos:funana}), we conclude that the claim in the lemma
is true if $\alpha=1$.

Second, we consider the case that $\alpha=2$. By combining the above
discussion and the proof used for Exercise 1.2.3 in page 34 of
Nualart~\cite{nua:malcal}, we know that the claim for $\alpha=2$ is
also true. $\Box$
\end{proof}
\begin{lemma}\label{chocone}
Consider each $F\in{\cal D}_{\infty}^{1,2}$. Then, for each
$c\in\{0,1,...\}$, $(i_{1},...,i_{p})\in{\cal I}^{c}$, and $x\in D$,
we have
\begin{eqnarray}
&&E\left[\left.D_{t}F_{i_{1}...i_{p}}^{(c)}(x)\right|{\cal
F}_{t}\right]\in\left(L^{2}([0,T]\times\Omega,C^{\infty}(D,R^{q}))\right)^{d},
\elabel{extendedsqrt}
\end{eqnarray}
and furthermore,
\begin{eqnarray}
&&F_{i_{1}...i_{p}}^{(c)}(x)=E\left[F_{i_{1}...i_{p}}^{(c)}(x)\right]
+\int_{0}^{T}E\left[\left.D_{t}F_{i_{1}...i_{p}}^{(c)}(x)\right|{\cal
F}_{t}\right]dW(t).\elabel{extendedCHO}
\end{eqnarray}
\end{lemma}
\begin{proof}
For each $F\in {\cal D}_{\infty}^{1,2}$ and $r\in\{1,...,q\}$, we
have the following calculation,
\begin{eqnarray}
&&E\left[\int_{0}^{T}\left\|E\left[\left.D_{t}F_{r,i_{1}...i_{p}}^{(c)}(x)\right|{\cal
F}_{t}\right]\right\|^{2}_{C^{\infty}(D,d)}dt\right]
\elabel{abnormine}\\
&\leq&\int_{0}^{T}E\left[\left\|D_{t}F_{r,i_{1}...i_{p}}^{(c)}(x)\right\|^{2}_{C^{\infty}(D,d)}\right]dt
\nonumber\\
&=&E\left[\int_{0}^{T}\left\|D_{t}F_{r,i_{1}...i_{p}}^{(c)}(x)\right\|^{2}_{C^{\infty}(D,d)}dt\right].
\nonumber
\end{eqnarray}
Thus, it follows from \eq{abnormine} that the claim in
\eq{extendedsqrt} is true. Furthermore, owing to the
Clark-Haussmann-Ocone formula (see, e.g., Aase {\em et
al.}~\cite{aasoks:whinoi}), we know that \eq{extendedCHO} holds.
$\Box$
\end{proof}
\begin{lemma}\label{equivm}
Let $Z\in L^{2}_{{\cal F},p}([t,T],C^{\infty}(D,R^{d}))$ with a
fixed $t\in[0,T]$ replacing $t=0$ in the previous discussion be such
that $F\in {\cal D}^{1,2}_{\infty}$ with $F$ defined by
\begin{eqnarray}
F^{(c)}_{i_{1}...i_{p}}(t,x)=\int_{t}^{T}Z^{(c)}_{i_{1}...i_{p}}(s,x)dW(s)
\elabel{xicI}
\end{eqnarray}
for each $x\in D$, $c\in\{0,1,...\}$, and $(i_{1},...,i_{p})\in{\cal
I}^{c}$. Then, $Z\in {\cal D}_{\infty}^{1,2}\cap
L^{2}_{1,2}(\Omega,(C^{\infty}(D,H))^{d})$. Furthermore, for each
$j\in\{1,...,d\}$,
\begin{eqnarray}
&&{\cal
D}_{\theta}^{j}F^{(c)}_{i_{1}...i_{p}}(t,x)=\left\{\begin{array}{ll}
       \int_{t}^{T}{\cal D}^{j}_{\theta}Z^{(c)}_{i_{1}...i_{p}}(s,x)dW(s)
       &\mbox{if}\;\;\theta\leq t,\\
       Z^{(c),j}_{i_{1}...i_{p}}(\theta,x)
       +\int_{\theta}^{T}{\cal D}^{j}_{\theta}Z^{(c)}_{i_{1}...i_{p}}(s,x)dW(s)
       &\mbox{if}\;\;\theta>t,
       \end{array}
\right.\elabel{oyequal}
\end{eqnarray}
where $Z^{(c),j}_{i_{1}...i_{p}}$ is the $j$th component of
$Z^{(c)}_{i_{1}...i_{p}}$.
\end{lemma}
\begin{proof}
First of all, if $Z\in {\cal D}_{\infty}^{1,2}\cap
L^{2}_{1,2}(\Omega,(C^{\infty}(D,H))^{d})$ and $F$ is defined by
\eq{xicI} for each fixed $t\in[0,T]$, it follows from Proposition
3.4 in Nualart and Pardoux~\cite{nuapar:stocal} that $F\in{\cal
D}^{1,2}_{\infty}$ and the claim in \eq{oyequal} holds for each
$x\in D$, $c\in\{0,1,...\}$, $(i_{1},...,i_{p})\in{\cal I}^{c}$, and
$j\in\{1,...,d\}$. Furthermore, owing to the Ito's isometry, we have
\begin{eqnarray}
&&\||F\||_{1,2}^{\infty,2}=\sum_{v=0}^{\infty}\xi(v)\left\|F\right\|^{v,2}_{1,2}
\elabel{xifinfty}
\end{eqnarray}
where
\begin{eqnarray}
&\left\||F\right\||^{v,2}_{1,2}
=E\left[\int_{t}^{T}\Lambda_{v}\left\|Z^{(c)}_{r,i_{1}...i_{p}}(s,x)
\right\|^{2}_{C^{v}(D,d)}ds+\int_{t}^{T}\int_{t}^{T}\Lambda_{v}\left\|{\cal
D}_{\theta}Z^{(c)}_{r,i_{1}...i_{p}}(s,x)\right\|^{2}_{C^{v}(D,dd)}
d\theta ds\right],\nonumber 
\end{eqnarray}
and $\Lambda_{v}$ is defined in \eq{notationI}. Therefore, by the
similar argument as used in the proof of Lemma 2.3 of Pardoux and
Peng~\cite{parpen:bacsto}, it suffices to show that the following
set for each fixed $t\in[0,T]$
\begin{eqnarray}
&&\left\{F\;\;\mbox{satisfying}\;\;\eq{xicI}\;\;\mbox{with}\;\;Z\in
{\cal D}_{\infty}^{1,2}\cap
L^{2}_{1,2}(\Omega,(C^{\infty}(D,H))^{d})\right\}\elabel{setI}
\end{eqnarray}
is dense in ${\cal D}_{\infty}^{1,2}\cap L^{2}_{{\cal
F}_{T}}(\Omega,C^{\infty}(D,R^{d}))$. In fact, it is a direct
conclusion of Lemma~\ref{closablelemma} and the fact that the set
defined in \eq{setI} contains the following set for each fixed
$t\in[0,T]$ owing to Lemma~\ref{chocone},
\begin{eqnarray}
&&\{F\in{\cal S}\cap L^{2}_{{\cal
F}_{T}}(\Omega,C^{\infty}(D,R^{d}))
\;\;\mbox{with}\;\;E[F^{(c)}_{i_{1}...i_{p}}(t,x)]=0
\nonumber\\
&&\;\;\mbox{for each}\;\;x\in
D\;\;c\in\{1,2,...\},\;(i_{1},...,i_{p})\in{\cal I}^{c}\}.
\nonumber
\end{eqnarray}
Hence, we complete the proof of the lemma. $\Box$
\end{proof}

Now, for each $t\in[t_{j_{0}-1},t_{j_{0}})$, $x\in D$, and
$v\in\{0,1,...,M\}$, we consider the following B-SPDE,
\begin{eqnarray}
V^{(v)}_{i_{1}...i_{p}}(t,x)&=&H^{(v)}_{i_{1}...i_{p}}(x)
+\int_{t}^{T}{\cal L}^{(v)}_{i_{1}...i_{p}}(s,x,V,\bar{V})ds
\elabel{pdsigmanVc}\\
&&+\int_{t}^{T}\left({\cal J}^{(v)}_{i_{1}...i_{p}}(s,x,V)
-\bar{V}^{(v)}_{i_{1}...i_{p}}(s,x)\right)dW(s), \nonumber
\end{eqnarray}
where $i_{1}+...+i_{p}=v$ and $i_{l}\in\{0,1,...,v\}$ with
$l\in\{1,...,p\}$. Then, we have the following lemma.
\begin{lemma}\label{differentiableVc}
Under conditions \eq{termI} and \eq{blipschitz}-\eq{blipschitzI},
there is a unique adapted and square-integrable solution
$(V^{(v)}_{i_{1}...i_{p}}(t,x),\bar{V}^{(v)}_{i_{1}...i_{p}}(t,x))$
to the B-SPDE in \eq{pdsigmanVc}.
\end{lemma}
\begin{proof}
The lemma is a direct conclusion of Theorem 2.1. $\Box$
\end{proof}
\begin{remark}\label{firstremark}
Note that, since the structures of the B-SPDEs displayed in
\eq{pdsigmanVc} are the same for all $v\in\{0,1,...,M\}$, we only
consider the case that $v=0$ in the rest of this subsection, i.e.,
the equation in \eq{bspdef}. Furthermore, for the time-inhomogeneous
B-SPDE in \eq{pdsigmanVc}, we can introduce an additional B-SPDE
through
\begin{eqnarray}
&&V^{0}(t,x)=T-\int_{t}^{T}ds.\elabel{additionaleq}
\end{eqnarray}
Obviously, $(V^{0}(t,x),\bar{V}^{0}(t,x))$ with $V^{0}(t,x)=t$ and
$\bar{V}^{0}(t,x)=(0,...,0)=\hat{0}$ (a $d$-dimensional zero row
vector) is the unique solution to the B-SPDE in \eq{additionaleq}.
Then, by combining \eq{additionaleq} and \eq{bspdef}, we can get a
$(q+1)$-dimensional B-SPDE,
\begin{eqnarray}
\;\;\;\;\;U(t,x)&=&\tilde{H}(x)+\int_{t}^{T}\tilde{{\cal L}}(x,U)ds
+\int_{t}^{T}\left(\tilde{{\cal J}}(x,U) -\bar{U}(s,x)\right)dW(s),
\elabel{ubspdef}
\end{eqnarray}
where
\begin{eqnarray}
\tilde{H}(x)&=&(T,H(x)')',\nonumber\\
U(t,x)&=&(V^{0}(t,x),V(t,x)')',\nonumber\\
\bar{U}(t,x)&=&(\bar{V}^{0}(t,x)',\bar{V}(t,x)')',\nonumber\\
\bar{{\cal L}}(x,U)&=&(-1,{\cal L}(x,U)')',\nonumber\\
\bar{J}(x,U)&=&(\hat{0}',{\cal J}(x,U)')'. \nonumber
\end{eqnarray}
Thus, without loss of generality and to be simple for notations, we
only consider the time-homogeneous case in \eq{bspdef} in the rest
of this section, i.e., the case corresponding to ${\cal
L}(s,x,V)={\cal L}(x,V)$ and ${\cal J}(s,x,V)={\cal J}(x,V)$.
\end{remark}

\subsection{B-SPDE with Malliavin Derivative Terminal
Condition}\label{malliavinterm}

First, consider a properly chosen number sequence
$\gamma=\{\gamma_{c},c=0,1,2,...\}$ satisfying
$0<\gamma_{0}<\gamma_{1}<....$ such that the discussions for Theorem
2.1 and Subsections~\ref{malliavinterm}-\ref{firstorderM} are
meaningful, which will be elaborated during the subsequent proof.
Second, we redefine the space in \eq{combanach} as follows,
\begin{eqnarray}
&&\;\;\;{\cal N}^{D}_{\gamma}[0,T]=D^{2}_{{\cal
F}}([0,T],C^{\infty}(D,R^{q\times d}))\times L^{2}_{{\cal
F},p}([0,T],C^{\infty}(D,R^{q\times d\times d})) \elabel{mcombanach}
\end{eqnarray}
endowed with the norm similarly defined as in
\eq{comnorm}-\eq{kcomnorm}. Then, we have the following lemma.
\begin{lemma}\label{limitadapted}
Under conditions as required in Theorem~\ref{errorIth} and with
Remark~\ref{firstremark}, if $(V(t,x),\bar{V}(t,x))$ $\in$ ${\cal
Q}_{{\cal F}}^{2}([0,T]\times D)$ is an adapted solution to
\eq{bspdef}, then the system of the following B-SPDEs has a unique
square-integrable adapted solution
$(V^{(c),\theta}_{i_{1}...i_{p}}(t,x)$,$\bar{V}^{(c),\theta}_{i_{1}...i_{p}}(t,x))$,
i.e., each component $(V^{j,(c),\theta}_{i_{1}...i_{p}}(t,x)$,
$\bar{V}^{j,(c),\theta}_{i_{1}...i_{p}}(t,x))$ satisfies ,
\begin{eqnarray}
\;\;\;\;\;\;V^{j,(c),\theta}_{i_{1}...i_{p}}(t,x)
&=&D^{j}_{\theta}H^{(c)}_{i_{1}...i_{p}}(x)\elabel{odmbspde}\\
&&+\int_{t}^{T}\sum_{l=0}^{k}\sum_{j_{1}+...+j_{p}=c+l}{\cal
L}^{(c+1)}_{i_{1}...i_{p},v_{j_{1}...j_{p}}}(x,V)V^{j,(c+l),\theta}_{j_{1}...j_{p}}(s,x)ds
\nonumber\\
&&+\int_{t}^{T}\sum_{l=0}^{m}\sum_{j_{1}+...+j_{p}=c+l}{\cal
L}^{(c+1)}_{i_{1}...i_{p},\bar{v}_{j_{1}...j_{p}}}(x,V)\bar{V}^{j,(c+l),\theta}_{j_{1}...j_{p}}(s,x)ds
\nonumber\\
&&+\int_{t}^{T}\sum_{l=0}^{n}\sum_{j_{1}+...+j_{p}=c+l}{\cal
J}^{(c+1)}_{i_{1}...i_{p},v_{j_{1}...j_{p}}}(x,V)V^{j,(c+l),\theta}_{j_{1}...j_{p}}(s,x)dW(s)
\nonumber\\
&&-\int_{t}^{T}\bar{V}^{j,(c),\theta}_{i_{1}...i_{p}}(s,x)dW(s)
\nonumber
\end{eqnarray}
for $j\in\{1,...,d\}$, $c\in\{0,1,...\}$, $(i_{1},...,i_{p})\in{\cal
I}^{c}$, $t\in[\theta,T]$, and $x\in D$. Furthermore, for each
$\delta\in(0,\frac{2}{3})$, there is a number sequence
$\gamma_{0}<\gamma_{1}<\cdot\cdot\cdot$ such that
\begin{eqnarray}
&\int_{0}^{T}\left\|(V^{\theta},\bar{V}^{\theta})\right\|^{2}_{{\cal
N}_{\gamma}^{D}[\theta,T]}d\theta<\frac{1}{1-\delta}
\left\|H\right\|^{\infty,2}_{1,2}+\frac{\delta
T^{2}}{1-\delta}<\infty. \elabel{ftheta}
\end{eqnarray}
\end{lemma}
\begin{proof}
First, we note that it follows from Theorem 2.1 and its proof that
there is a unique adapted solution $(V(t,x),\bar{V}(t,x))\in{\cal
Q}_{{\cal F}}^{2}([0,T]\times D)$ to \eq{bspdef}. Thus, we know that
there is no explosion time for the process $(V(t,x),\bar{V}(t,x))$
over time interval $[0,T]$. Furthermore, for each $x\in D$,
$V(\cdot,x)$ is a c\`adl\`ag process and $\bar{V}(\cdot,x)$ is its
corresponding predictable process. Then, it follows from Theorem 3
in page 4 of Protter~\cite{pro:stoint} and Remark 5.1 in page 21 of
Ikeda and Watanabe~\cite{ikewat:stodif} that
\begin{eqnarray}
&&\tau_{w}\equiv T\wedge\inf\left\{t>0,\|V(t)\|_{C^{\infty}(D,q)}
+\left\|\bar{V}(t)\right\|_{C^{\infty}(D,qd)}>w\right\}
\elabel{stoptime}
\end{eqnarray}
is a sequence of nondecreasing $\{{\cal F}_{t}\}$-stopping times and
satisfies $\tau_{w}\rightarrow T$ a.s. as $w\rightarrow\infty$ along
$w\in\{0,1,...\}$.

Now, for all $j\in\{1,...,d\}$, $c\in\{0,1,...\}$,
$(i_{1},...,i_{p})\in{\cal I}^{c}$, $t\in[\theta,\tau_{w}]$, and
$x\in D$, we define the following system of B-SPDEs,
\begin{eqnarray}
\;\;\;\;\;\;V^{j,(c),\theta}_{i_{1}...i_{p}}(t,x)
&=&E\left[\left.D^{j}_{\theta}H^{(c)}_{i_{1}...i_{p}}(x)\right|{\cal
F}_{\tau_{w}}\right]
\elabel{odmbspdeI}\\
&&+\int_{t\wedge\tau_{w}}^{\tau_{w}}\sum_{l=0}^{k}\sum_{j_{1}+...+j_{p}=c+l}{\cal
L}^{(c+1)}_{i_{1}...i_{p},v_{j_{1}...j_{p}}}(x,V)V^{j,(c+l),\theta}_{j_{1}...j_{p}}(s,x)ds
\nonumber\\
&&+\int_{t\wedge\tau_{w}}^{\tau_{w}}\sum_{l=0}^{m}\sum_{j_{1}+...+j_{p}=c+l}{\cal
L}^{(c+1)}_{i_{1}...i_{p},\bar{v}_{j_{1}...j_{p}}}(x,V)\bar{V}^{j,(c+l),\theta}_{j_{1}...j_{p}}(s,x)ds
\nonumber\\
&&+\int_{t\wedge\tau_{w}}^{\tau_{w}}\sum_{l=0}^{n}\sum_{j_{1}+...+j_{p}=c+l}{\cal
J}^{(c+1)}_{i_{1}...i_{p},v_{j_{1}...j_{p}}}(x,V)V^{j,(c+l),\theta}_{j_{1}...j_{p}}(s,x)dW(s)
\nonumber\\
&&-\int_{t\wedge\tau_{w}}^{\tau_{w}}\bar{V}^{j,(c),\theta}_{i_{1}...i_{p}}(s,x)dW(s).
\nonumber
\end{eqnarray}
Then, over the random time interval $[\theta,\tau_{w}]$, the B-SPDEs
in \eq{odmbspdeI} satisfy conditions \eq{blipschitz} and
\eq{blipschitzI}. Thus, by slightly generalizing the discussions in
proving Theorem~\ref{bsdeyI} and Yong and Zhou~\cite{yonzho:stocon},
we know that \eq{odmbspdeI} has a unique adapted solution
$(V^{\theta,w},\bar{V}^{\theta,w})$ in ${\cal
N}^{D}_{\gamma}[\theta,\tau_{w}]$. Furthermore, the solution has the
following infinite-dimensional vector form,
\begin{eqnarray}
&&\left\{(V^{(c),\theta,w}_{r,i_{1}...i_{p}}(t,x),
\bar{V}^{(c),\theta,w}_{r,i_{1}...i_{p}}(t,x)),
r\in\{1,...,q\},c\in\{0,1,...\},(i_{1},...,i_{p})\in{\cal
I}^{c}\right\}.\elabel{defvt}
\end{eqnarray}
Thus, by \eq{termI}, \eq{blipschitzoI}-\eq{blipschitzI}, the
It$\hat{o}$'s formula, and the similar technique in the proof for
the claims in \eq{lastine} and \eq{clastine}, we know that
\begin{eqnarray}
&&\left\|(V^{\theta,w},\bar{V}^{\theta,w})\right\|^{2}_{{\cal
N}_{\gamma_{c},c}^{D}[\theta,\tau_{w}]}
\elabel{vtheta}\\
&\leq&E\left[\bar{\Lambda}_{c}\left\|D_{\theta}H^{(v)}_{i_{1}...i_{p}}
\right\|_{C^{\infty}(D,q)}^{2}\right]+\hat{\gamma}_{c}K^{1}_{d,c}
\left((T-\theta)\delta_{0c}
+\left\|(V^{\theta,w},\bar{V}^{\theta,w})\right\|^{2}_{{\cal
N}_{\gamma_{c},c+2k}^{D}[\theta,\tau_{w}]}\right) \nonumber
\end{eqnarray}
for each $c\in\{0,1,...\}$. The notation $\bar{\Lambda}_{c}$ in
\eq{vtheta} is defined by
\begin{eqnarray}
&&\bar{\Lambda}_{c}=\max_{v\in\{0,1,...,c\}}\max_{(i_{1},...,i_{p})\in{\cal
I}^{c}},\elabel{notationII}
\end{eqnarray}
and $\delta_{0c}$ is defined in \eq{blipschitzoI}-\eq{blipschitzI}.
Furthermore, $K^{1}_{d,c}$ is some nonnegative constant depending
only on $c$, $T$ and the region $D$, which satisfies
$K^{1}_{d,c}\geq K_{d,c}$ (that is used in \eq{clastine}). In
addition, $\hat{\gamma}_{c}$ is a nonnegative constant depending on
$\gamma_{c}$ and can be arbitrarily chosen by suitably managing the
number sequence $\gamma_{0}<\gamma_{1}<\cdot\cdot\cdot$ such that
\begin{eqnarray}
&&\hat{\gamma}_{c}K^{1}_{d,c}((c+1)^{10}(c+2)^{10}...(c+2k)^{10})(\eta(c+1)
\eta(c+2)...\eta(c+2k))e^{2k}\leq\delta \elabel{hatgma}
\end{eqnarray}
for some constant $\delta\in(0,2/3)$. Therefore, we have
\begin{eqnarray}
&&\left\|(V^{\theta,w},\bar{V}^{\theta,w})\right\|_{{\cal
N}_{\gamma}^{D}[\theta,\tau_{w}]}\nonumber\\
&\leq&\sum_{c=1}^{\infty}\xi(c) E\left[\bar{\Lambda}_{c}
\left\|D_{\theta}H^{(v)}_{i_{1}...i_{p}}\right\|_{C^{\infty}(D,q)}^{2}\right]
+\delta
T+\delta\left\|(V^{\theta,w},\bar{V}^{\theta,w})\right\|_{{\cal
N}_{\gamma}^{D}[\theta,\tau_{w}]}. \nonumber
\end{eqnarray}
Thus, it follows from conditions \eq{termI}-\eq{termIII} that
\begin{eqnarray}
\;\;\left\|(V^{\theta,w},\bar{V}^{\theta,w})\right\|^{2}_{{\cal
N}_{\gamma}^{D}[\theta,\tau_{w}]}
&\leq&\frac{1}{1-\delta}\sum_{c=1}^{\infty}\xi(c)E\left[\bar{\Lambda}_{c}
\left\|D_{\theta}H^{(v)}_{i_{1}...i_{p}}\right\|_{C^{\infty}(D,q)}^{2}\right]
+\frac{\delta T}{1-\delta}
\elabel{vvbar}\\
&<&\infty.
\nonumber
\end{eqnarray}

Second, we use
$\Pi^{\theta,w}(t,x)\equiv(V^{\theta,w}(t,x),\bar{V}^{\theta,w}(t,x))$
for $t\leq\tau_{w}$ and $x\in D$ to denote the unique adapted
solution to the system in \eq{odmbspdeI} for each $w\in\{1,2,...\}$.
Then, it follows from the Ito's formula, conditions
\eq{blipschitz}-\eq{blipschitzI}, and the similar proof for
\eq{vvbar} that
\begin{eqnarray}
&&\left\|\Pi^{\theta,w_{1}}-\Pi^{\theta,w_{2}}\right\|^{2}_{{\cal
N}_{\gamma}^{D}[\theta,T]}\elabel{cauchyr}\\
&\leq&\frac{1}{1-\delta}\sum_{c=1}^{\infty}\xi(c)
E\left[\bar{\Lambda}_{c}\left\|E\left[\left.D_{\theta}H^{(v)}_{i_{1}...i_{p}}\right|{\cal
F}_{\tau_{w_{1}}}\right]-E\left[\left.D_{\theta}H^{(v)}_{i_{1}...i_{p}}\right|{\cal
F}_{\tau_{w_{2}}}\right]\right\|^{2}_{C^{\infty}(D,q)}\right]
\nonumber\\
&&+\frac{\delta}{1-\delta}E\left[\tau_{w_{2}}-\tau_{w_{1}}\right]
\nonumber\\
&\leq&\frac{1}{1-\delta}\sum_{c=1}^{\infty}\xi(c)
\left(E\left[\bar{\Lambda}_{c}\left\|E\left[\left.D_{\theta}H^{(v)}_{i_{1}...i_{p}}\right|{\cal
F}_{\tau_{w_{1}}}\right]-D_{\theta}H^{(v)}_{i_{1}...i_{p}}\right\|^{2}_{C^{\infty}(D,q)}\right]\right.
\nonumber\\
&&+\left.E\left[\bar{\Lambda}_{c}\left\|E\left[\left.D_{\theta}H^{(v)}_{i_{1}...i_{p}}\right|{\cal
F}_{\tau_{w_{1}}}\right]-D_{\theta}H^{(v)}_{i_{1}...i_{p}}\right\|^{2}_{C^{\infty}(D,q)}\right]\right)
+\frac{\delta}{1-\delta}E\left[\tau_{w_{2}}-\tau_{w_{1}}\right]
\nonumber\\
&\rightarrow&0 \nonumber
\end{eqnarray}
as $w_{1},w_{2}\rightarrow\infty$ along $w_{1},w_{2}\in\{1,2,...\}$.
Note that the last claim of \eq{cauchyr} follows from
\eq{termI}-\eq{termIII}
and the Martingale
convergence theorem (see, e.g., Page 8 of
Protter~\cite{pro:stoint}). Furthermore, in the proof of
\eq{cauchyr}, we also used the fact that
\begin{eqnarray}
&&V^{\theta,w}(t,x)=E\left[\left.D_{\theta}H^{(v)}_{i_{1}...i_{p}}(x)\right|{\cal
F}_{\tau_{w_{1}}}\right]\;\;\;\mbox{for each}\;\;\;t\in[\tau_{w},T].
\end{eqnarray}
Thus, from \eq{cauchyr}, we know that
$\{\Pi^{\theta,w},w\in\{1,2,...\}\}$ is a cauchy sequence in ${\cal
N}_{\gamma}^{D}[\theta,T]$. Hence, there is a $\Pi^{\theta}\in{\cal
N}_{\gamma}^{D}[\theta,T]$ such that
\begin{eqnarray}
&&\Pi^{\theta,w}\rightarrow\Pi^{\theta}\;\;\mbox{as}\;\;w\rightarrow\infty.
\elabel{conpi}
\end{eqnarray}
In addition, we claim that $\Pi^{\theta}$ is the unique
square-integrable adapted solution to the system of B-SPDEs in
\eq{odmbspde}.

In fact, since $\Pi^{\theta,w}$ is a solution satisfying
\eq{odmbspdeI} for each $w\in\{1,2,...\}$, it follows from the Ito's
isometry, Holder's inequality, the similar ideas as used for
\eq{cauchyr} and the proof of Theorem 5.1.2 in page 68 of
$\emptyset$ksendal~\cite{oks:stodif} that $\Pi^{\theta}$ is a
square-integrable adapted solution to the system of B-SPDEs in
\eq{odmbspde}. Next, suppose that $\Pi^{\theta}_{1}$ and
$\Pi^{\theta}_{2}$ are two required solutions to the system in
\eq{odmbspde}. Then, $\Pi^{\theta}_{1}-\Pi^{\theta}_{2}$ is a
square-integrable adapted solution to the system in \eq{odmbspde}
with terminal value 0. Thus, $\Pi^{\theta}_{1}-\Pi^{\theta}_{2}$ is
the unique square-integrable adapted solution to the system in
\eq{odmbspdeI} with terminal value 0 over each random interval
$[0,\tau_{w}]$ for $w\in\{1,2,...\}$. Since $\tau_{w}\rightarrow T$
as $w\rightarrow\infty$, we know that
$\Pi^{\theta}_{1}=\Pi^{\theta}_{2}$ a.s.

Finally, it follows from \eq{conpi} and \eq{vvbar} that the claim in
\eq{ftheta} is true. Thus, we complete the proof of
Lemma~\ref{limitadapted}. $\Box$
\end{proof}

\subsection{First-Order Malliavin Derivative Based
B-SPDE}\label{firstorderM}

First, let $L^{\infty,2}_{\alpha,2}([0,T]\times\Omega,
(C^{\infty}(D,H))^{q})$ represent the set of $H^{q}$-valued
progressively measurable processes $\{\zeta(t,x,\omega),0\leq t\leq
T,\omega\in\Omega\}$ for each $x\in D$ such that
\begin{itemize}
\item For a.e. $t\in[0,T]$, $\zeta(t,\cdot,\cdot)\in{\cal D}^{\alpha,2}_{\infty}$;
\item $(t,x,\omega)\rightarrow{\cal D}\zeta_{i_{1}...i_{p}}^{(c)}(t,x,\omega)\in
(L^{2}([0,T]\times\Omega,C^{\infty}(D,R^{q})))^{d}$ admits a
progressively measurable version for each $x\in D$,
$c\in\{0,1,...\}$, $(i_{1},...,i_{p})\in{\cal I}^{c}$ if $\alpha=1$.
In addition, $(t,x,\omega)\rightarrow{\cal D}{\cal
D}\zeta_{i_{1}...i_{p}}^{(c)}(t,x,\omega)\in
(L^{2}([0,T]\times\Omega,C^{\infty}(D,R^{q})))^{d\times d}$ also
admits a progressively measurable version if $\alpha=2$;
\item The following norm is defined,
\begin{eqnarray}
\||\zeta\||^{\infty,2}_{\alpha,2}=\sum_{v=0}^{\infty}\xi(v)
\||\zeta\||^{v,2}_{\alpha,2}<\infty, \nonumber
\end{eqnarray}
where, for each $v\in\{0,1,...\}$,
\begin{eqnarray}
\||\zeta\||^{v,2}_{1,2}=&E\left[\int_{0}^{T}\Lambda_{v}
\left\|\zeta^{(c)}_{r,i_{1}...i_{p}}(t)\right\|^{2}_{C^{v}(D,1)}dt
+\int_{0}^{T}\int_{0}^{T}\Lambda_{v}\left\|{\cal
D}_{\theta_{1}}\zeta^{(c)}_{r,i_{1}...i_{p}}(t)\right\|^{2}_{C^{v}(D,
d)} d\theta_{1}dt\right], \nonumber\\
\||\zeta\||^{v,2}_{2,2}
=&\left\|\zeta\right\|^{v,2}_{1,2}+E\left[\int_{0}^{T}\int_{0}^{T}\int_{0}^{T}\Lambda_{v}\left\|{\cal
D}_{\theta_{2}}{\cal
D}_{\theta_{1}}\zeta^{(c)}_{r,i_{1}...i_{p}}(t)\right\|^{2}_{C^{v}(D,d\times
d)}d\theta_{1}d\theta_{2}dt\right]. \nonumber
\end{eqnarray}
\end{itemize}
Then, we have the following lemma.
\begin{lemma}\label{masolution}
Under conditions as required in Theorem~\ref{errorIth} and with
Remark~\ref{firstremark}, if $(V(t,x),\bar{V}(t,x))$ $\in$ ${\cal
Q}_{{\cal F}}^{2}([0,T]\times D)$ is the adapted solution to
\eq{bspdef}, then
\begin{eqnarray}
&(V(t,x),\bar{V}(t,x))\in
L^{\infty,2}_{1,2}([0,T]\times\Omega,(C^{\infty}(D,H))^{q\times
q\times d}). \nonumber
\end{eqnarray}
Furthermore, for each $1\leq j\leq d$ and $x\in D$, a version of the
infinite-dimensional vector process
\begin{eqnarray}
&\{(D^{j}_{\theta}V^{(c)}_{i_{1}...i_{p}}(t,x),
D^{j}_{\theta}\bar{V}^{(c)}_{i_{1}...i_{p}}(t,x)):0\leq\theta,t\leq
T,c\in\{0,1,...\}, (i_{1},...,i_{p})\in{\cal I}^{c}\}
\nonumber
\end{eqnarray}
is the solution of the following system of Malliavin derivative
based B-SPDEs under random environment,
\begin{eqnarray}
D^{j}_{\theta}V^{(c)}_{i_{1}...i_{p}}(t,x)
&=&D^{j}_{\theta}H^{(c)}_{i_{1}...i_{p}}(x)\elabel{mbspde}\\
&&+\int_{t}^{T}\sum_{l=0}^{k}\sum_{j_{1}+...+j_{p}=c+l}{\cal
L}^{(c+1)}_{i_{1}...i_{p},v_{j_{1}...j_{p}}}(x,V)D^{j}_{\theta}V^{(c+l)}_{j_{1}...j_{p}}(s,x)ds
\nonumber\\
&&+\int_{t}^{T}\sum_{l=0}^{m}\sum_{j_{1}+...+j_{p}=c+l}{\cal
L}^{(c+1)}_{i_{1}...i_{p},\bar{v}_{j_{1}...j_{p}}}(x,V)D^{j}_{\theta}\bar{V}^{(c+l)}_{j_{1}...j_{p}}(s,x)ds
\nonumber\\
&&+\int_{t}^{T}\sum_{l=0}^{n}\sum_{j_{1}+...+j_{p}=c+l}{\cal
J}^{(c+1)}_{i_{1}...i_{p},v_{j_{1}...j_{p}}}(x,V)D^{j}_{\theta}V^{(c+l)}_{j_{1}...j_{p}}(s,x)dW(s)
\nonumber\\
&&-\int_{t}^{T}D^{j}_{\theta}\bar{V}^{(c)}_{i_{1}...i_{p}}(s,x)dW(s),
\nonumber
\end{eqnarray}
where $j_{1},...,j_{p}$ are nonnegative integers, and for $0\leq
t<\theta\leq T$,
\begin{eqnarray}
&&D^{j}_{\theta}V^{(c)}_{i_{1}...i_{p}}(t,x)=0,
\;\;\;D^{j}_{\theta}\bar{V}^{(c)}_{i_{1}...i_{p}}(t,x)=0.
\elabel{maequalzero}
\end{eqnarray}
In addition, let ``$=^{d}$" denote ``equal in distribution", then
\begin{eqnarray}
&&\left\{\bar{V}^{(c)}_{i_{1}...i_{p}}(t,x),t\in[0,T],c\in\{0,1,...\},(i_{1},...,i_{p})\in{\cal I}^{c},
x\in D\right\}\elabel{equaldis}\\
&=^{d}&\left\{D_{t}V^{(c)}_{i_{1}...i_{p}}(t,x)+{\cal
J}^{(c)}_{i_{1}...i_{p}}(x,V),t\in[0,T],c\in\{0,1,...\},(i_{1},...,i_{p})\in{\cal
I}^{c},x\in D\right\}. \nonumber
\end{eqnarray}
\end{lemma}
\begin{proof}
First, it follows from Theorem 2.1 and its proof that there is a
unique adapted solution $(V(t,x),\bar{V}(t,x))\in{\cal Q}_{{\cal
F}}^{2}([0,T]\times D)$ to \eq{bspdef}. Furthermore, it can be
approximated by a sequence of $(V^{i}(t,x),\bar{V}^{i}(t,x))\in{\cal
M}_{\gamma}^{D}[0,T]$ with $i\in\{0,1,...\}$, satisfying,
\begin{eqnarray}
V^{0}(t,x)&=&\bar{V}^{0}(t,x)=0\elabel{minitialV}\\
V^{(c),i+1}_{i_{1}...i_{p}}(t,x)&=&H^{(c)}_{i_{1}...i_{p}}(x)+\int_{t}^{T}{\cal
L}^{(c)}_{i_{1}...i_{p}}(x,V^{i})ds \elabel{msigmanV}\\
&&+\int_{t}^{T}\left({\cal J}^{(c)}_{i_{1}...i_{p}}(x,V^{i})
-\bar{V}^{(c),i+1}_{i_{1}...i_{p}}(s,x)\right)dW(s)\nonumber
\end{eqnarray}
for all $t\in[0,T]$, $x\in D$, and $(i_{1},...,i_{p})\in{\cal
I}^{c}$.

Now, by induction in terms of $i\in\{0,1,...\}$, we can show that
\begin{eqnarray}
&(V^{i}(t,x),\bar{V}^{i}(t,x))\in
L^{\infty,2}_{1,2}([0,T]\times\Omega,(C^{\infty}(D,H))^{q\times
q\times d}). \nonumber
\end{eqnarray}
Equivalently, if
\begin{eqnarray}
&(V^{i}(t,x),\bar{V}^{i}(t,x))\in
L^{\infty,2}_{1,2}([0,T]\times\Omega,(C^{\infty}(D,H))^{q\times
q\times d}) \nonumber
\end{eqnarray}
for any $i\in\{0,1,...\}$, we need to prove that
\begin{eqnarray}
&(V^{i+1}(t,x),\bar{V}^{i+1}(t,x))\in
L^{\infty,2}_{1,2}([0,T]\times\Omega,(C^{\infty}(D, H))^{q\times
q\times d}). \nonumber
\end{eqnarray}
In fact, since
\begin{eqnarray}
&&H(x)+\int_{t}^{T}{\cal L}(x,V^{i})ds\in{\cal D}^{1,2}_{\infty},
\elabel{dI-II}
\end{eqnarray}
it follows from Lemma~\ref{chocone} and Lemma~\ref{equivm} that
\begin{eqnarray}
V^{i+1}(t,x) &=&E\left[\left.H(x)+\int_{t}^{T}{\cal
L}(x,V^{i})ds\right|{\cal F}_{t}\right]\in {\cal
D}^{1,2}_{\infty}\elabel{dI-III}.
\end{eqnarray}
Thus, by using \eq{msigmanV}-\eq{dI-III} and Lemma~\ref{equivm}, we
know that
\begin{eqnarray}
&&{\cal J}(x,V^{i}) -\bar{V}^{i+1}(s,x)\in{\cal
D}^{1,2}_{\infty}.\nonumber
\end{eqnarray}
Hence, by chain rule for Malliavin calculus, we have that
\begin{eqnarray}
&&\bar{V}^{i+1}(s,x)\in{\cal D}^{1,2}_{\infty}.\nonumber
\end{eqnarray}
Therefore, for each $0\leq\theta\leq t$ and $j\in\{1,...,d\}$, it
follows from chain rule for Malliavin calculus and
Lemma~\ref{equivm} that
\begin{eqnarray}
\;\;\;\;\;\;D^{j}_{\theta}V^{(c),i+1}_{i_{1}...i_{p}}(t,x)
&=&D^{j}_{\theta}H^{(c)}_{i_{1}...i_{p}}(x)\elabel{dmbspde}\\
&&+\int_{t}^{T}\sum_{l=0}^{k}\sum_{j_{1}+...+j_{p}=c+l}{\cal
L}^{(c+1)}_{i_{1}...i_{p},v_{j_{1}...j_{p}}}(x,V^{i})D^{j}_{\theta}V^{(c+l),i}_{j_{1}...j_{p}}(s,x)ds
\nonumber\\
&&+\int_{t}^{T}\sum_{l=0}^{m}\sum_{j_{1}+...+j_{p}=c+l}{\cal
L}^{(c+1)}_{i_{1}...i_{p},\bar{v}_{j_{1}...j_{p}}}(x,V^{i})D^{j}_{\theta}\bar{V}^{(c+l),i}_{j_{1}...j_{p}}(s,x)ds
\nonumber\\
&&+\int_{t}^{T}\sum_{l=0}^{n}\sum_{j_{1}+...+j_{p}=c+l}{\cal
J}^{(c+1)}_{i_{1}...i_{p},v_{j_{1}...j_{p}}}(x,V^{i})D^{j}_{\theta}V^{(c+l),i}_{j_{1}...j_{p}}(s,x)dW(s)
\nonumber\\
&&-\int_{t}^{T}D^{j}_{\theta}\bar{V}^{(c),i+1}_{i_{1}...i_{p}}(s,x)dW(s).
\nonumber
\end{eqnarray}
Furthermore, it follows from the proof of Lemma~\ref{martindecom}
that
\begin{eqnarray}
\left\{(D^{j}_{\theta}V^{(c),i+1}_{i_{1}...i_{p}}(t,x),
D^{j}_{\theta}\bar{V}^{(c),i+1}_{i_{1}...i_{p}}(t,x)),
i\in\{1,2,...\},j\in\{1,...,d\},c\in\{0,1,...\},
(i_{1},...,i_{p})\in{\cal I}^{c}\right\}\nonumber
\end{eqnarray}
is the unique adapted solution to the system in \eq{dmbspde}. In
addition, it follows from the similar proof of
Lemma~\ref{differentiableV} that this solution is continuous with
respect to $x\in D$.

Next, we show that $(V^{i}(t,x),\bar{V}^{i}(t,x))$ converges in
$L^{\infty,2}_{1,2}([0,T]\times\Omega,(C^{\infty}(D,H))^{q\times
q\times d})$. In particular, we have the convergence of their
Malliavin derivatives as $i\rightarrow\infty$ as follows,
\begin{eqnarray}
&&\left\{(D_{\theta}V^{(c),i}_{i_{1}...i_{p}}(t,x),D_{\theta}\bar{V}^{(c),i}_{i_{1}...i_{p}}(t,x)),
c\in\{0,1,...\},(i_{1},...,i_{p})\in{\cal
I}^{c},t\in[\theta,T],x\in D\right\}\elabel{limitm}\\
&&\rightarrow\left\{(V^{(c),\theta}_{i_{1}...i_{p}}(t,x),\bar{V}^{(c),\theta}_{i_{1}...i_{p}}(t,x)),
c\in\{0,1,...\},(i_{1},...,i_{p})\in{\cal I}^{c},t\in[\theta,T],x\in
D\right\},\nonumber
\end{eqnarray}
where each component
$(V^{j,(c),\theta}_{i_{1}...i_{p}}(t,x)$,$\bar{V}^{j,(c),\theta}_{i_{1}...i_{p}}(t,x))$
of the limit
$(V^{(c),\theta}_{i_{1}...i_{p}}(t,x)$,$\bar{V}^{(c),\theta}_{i_{1}...i_{p}}(t,x))$
with $j\in\{1,...,d\}$ is the unique adapted solution to the
B-SSPDEs in \eq{odmbspde} for all $c\in\{0,1,...\}$,
$(i_{1},...,i_{p})\in{\cal I}^{c}$, $t\in[\theta,T]$, and $x\in D$
owing to Lemma~\ref{limitadapted}.

Now, by applying conditions \eq{termI},
\eq{blipschitz}-\eq{blipschitzI}, the It$\hat{o}$'s formula, and the
similar technique used in the proof of Lemma~\ref{lemmathree}, we
have that
\begin{eqnarray}
&&\left\|({\cal D}_{\theta}V^{i+1}-V^{\theta},{\cal
D}_{\theta}\bar{V}^{i+1}-\bar{V}^{\theta})\right\|_{{\cal
N}_{\gamma_{c},c}}^{2}\leq\hat{\gamma}_{c}K^{1}_{d,c}E\left[\int_{\theta}^{T}
\left(\alpha^{i}(s)+\beta^{i}(s)\right)e^{2\gamma_{c}s}ds\right]
\elabel{firstine}
\end{eqnarray}
for all $c\in\{0,1,...\}$ with each given $i\in\{0,1,...\}$. The
notation $K^{1}_{d,c}$ is some nonnegative constant depending only
on $c$, $D$, and $T$, $\hat{\gamma}_{c}$ is taken and explained as
in \eq{hatgma}. The functions $\alpha^{i}(s)$ and $\beta^{i}(s)$ are
respectively given by
\begin{eqnarray}
\alpha^{i}(s)
&=&\left(1+\left\|V^{i}(s)\right\|^{2}_{C^{c+k+1}(D,q)}\right)\left\|{\cal
D}_{\theta}V^{i}(s)-V^{\theta}(s)\right\|^{2}_{C^{c+k+1}(D,qd)}
\nonumber\\
&&+\left(1+\left\|\bar{V}^{i}(s)\right\|^{2}_{C^{c+k+1}(D,qd)}\right)
\left\|{\cal
D}_{\theta}\bar{V}^{i}(s)-\bar{V}^{\theta}(s)\right\|^{2}_{C^{c+k+1}(D,qdd)},
\nonumber\\
\beta^{i}(s)&=&\left\|V^{\theta}(s)\right\|^{2}_{C^{c+k+1}(D,qd)}
\left(1+\left\|V^{i}(s)-V(s)\right\|^{2}_{C^{c+k+1}(D,q)}\right)
\nonumber\\
&&+\left\|\bar{V}^{\theta}(s)\right\|^{2}_{C^{c+k+1}(D,qdd)}
\left(1+\left\|\bar{V}^{i}(s)-\bar{V}(s)\right\|^{2}_{C^{c+k+1}(D,qd)}\right).
\nonumber
\end{eqnarray}
Thus, by \eq{ftheta}, \eq{firstine}, and the fact that
$|ab|\leq\frac{1}{2}(a^{2}+b^{2})$ for any two real numbers $a$ and
$b$, we have that
\begin{eqnarray}
&&\int_{0}^{T}\left\|({\cal D}_{\theta}V^{i+1}-V^{\theta},{\cal
D}_{\theta}\bar{V}^{i+1}-\bar{V}^{\theta})\right\|_{{\cal
N}_{\gamma}^{D}[\theta,T]}^{2}d\theta
\elabel{secondine}\\
&\leq&\delta
T\left\|(V^{i}-V,\bar{V}^{i}-\bar{V})\right\|^{2}_{{\cal
N}_{\gamma}^{D}[\theta,T]}+\delta
T\left\|(V,\bar{V})\right\|^{2}_{{\cal N}_{\gamma}^{D}}+
\frac{3\delta}{2}\int_{0}^{T}\left\|(V^{\theta},\bar{V}^{\theta})\right\|^{2}_{{\cal
N}_{\gamma}^{D}[\theta,T]}d\theta
\nonumber\\
&&+\frac{3\delta}{2}E\left[\int_{0}^{T}\left\|({\cal
D}_{\theta}V^{i}-V^{\theta},{\cal
D}_{\theta}\bar{V}^{i}-\bar{V}^{\theta})\right\|_{{\cal
N}_{\gamma}^{D}[\theta,T]}^{2}d\theta\right] \nonumber\\
&\leq&(\delta+...+\delta^{i})K_{1}+\delta^{i}\int_{0}^{T}\left\|({\cal
D}_{\theta}V^{0}-V^{\theta},{\cal
D}_{\theta}\bar{V}^{0}-\bar{V}^{\theta})\right\|_{{\cal
N}_{\gamma}^{D}[\theta,T]}^{2}d\theta
\nonumber\\
&\leq&\frac{\delta K_{1}}{1-\delta}+\delta^{i}K_{2}\nonumber
\end{eqnarray}
where $K_{1}$ and $K_{2}$ are some nonnegative constants. Since
$e^{2\gamma_{c}t}>1$ for all $c\in\{0,1,...\}$, we know that
\begin{eqnarray}
&&\sum_{v=0}^{\infty}\xi(v)E\left[\int_{0}^{T}\int_{0}^{T}\Lambda_{v}\left\|({\cal
D}_{\theta}V^{i+1}-V^{\theta},{\cal
D}_{\theta}\bar{V}^{i+1}-\bar{V}^{\theta})\right\|_{C^{v}(D,qd\times
qdd)}^{2}d\theta
dt\right]\elabel{madip}\\
&&\leq\frac{\delta K_{1}}{1-\delta}+\delta^{i}K_{2}\nonumber\\
&&\rightarrow 0\nonumber
\end{eqnarray}
by letting $i\rightarrow\infty$ first and $\delta\rightarrow 0$
second since $\delta\in(0,1)$. Thus, by \eq{madip} and the factor
that $e^{2\gamma_{c}t}>1$ again, we have
\begin{eqnarray}
&\||(V^{i},\bar{V}^{i})-(V,\bar{V})\||^{\infty,2}_{1,2}\rightarrow
0\;\;\mbox{as}\;\;i\rightarrow\infty.
\end{eqnarray}
Thus, we know that $(V^{i},\bar{V}^{i})$ with Malliavin derivative
$({\cal D}_{\theta}V^{i}$, ${\cal D}_{\theta}\bar{V}^{i})$ converges
to $(V,\bar{V})$ with Malliavin derivative $(V^{\theta}$,
$\bar{V}^{\theta})$ in
$L^{\infty,2}_{1,2}([0,T]\times\Omega,(C^{\infty}(D,H))^{q\times
q\times d})$ as $i\rightarrow\infty$. Hence, a version of the
following infinite-dimensional vector process
\begin{eqnarray}
&\{(D^{j}_{\theta}V^{(c)}_{i_{1}...i_{p}}(t,x),
D^{j}_{\theta}\bar{V}^{(c)}_{i_{1}...i_{p}}(t,x)):0\leq\theta,t\leq
T,c\in\{0,1,...\},j\in\{1,...,d\},(i_{1},...,i_{p})\in{\cal I}^{c}\}
\nonumber
\end{eqnarray}
is given by \eq{mbspde}.

Finally, for the considered version, the claims in \eq{maequalzero}
of Lemma~\ref{masolution} are follows from the fact that
$(V,\bar{V})$ is an adapted solution to the B-SPDE displayed in
\eq{bspdef} and Corollary 1.2.1 in page 34 and its related remark in
page 42 of Nualart~\cite{nua:malcal}. Furthermore, the claims in
\eq{equaldis} are justified as follows. Since, for $t\leq u$, we
have that
\begin{eqnarray}
V^{(c)}_{i_{1}...i_{p}}(u,x)&=&V^{(c)}_{i_{1}...i_{p}}(t,x)-\int_{t}^{u}{\cal
L}^{(c)}_{i_{1}...i_{p}}(x,V)ds \elabel{fmsigmanV}\\
&&-\int_{t}^{u}\left({\cal J}^{(c)}_{i_{1}...i_{p}}(x,V)
-\bar{V}^{(c)}_{i_{1}...i_{p}}(s,x)\right)dW(s)\nonumber
\end{eqnarray}
for all $x\in D$, $c\in\{0,1,...\}$, and $(i_{1},...,i_{p})\in{\cal
I}^{c}$. Then, it follows from Lemma~\ref{equivm} that, for
$j\in\{1,...,d\}$ and $t<\theta\leq u$,
\begin{eqnarray}
\;\;\;\;\;\;D^{j}_{\theta}V^{(c)}_{i_{1}...i_{p}}(u,x)
&=&\bar{V}^{(c),j}_{i_{1}...i_{p}}(\theta,x)-{\cal
J}^{(c),j}_{i_{1}...i_{p}}(x,V)
\elabel{fdmbspde}\\
&&-\int_{\theta}^{u}\sum_{l=0}^{k}\sum_{j_{1}+...+j_{p}=c+l}{\cal
L}^{(c+1)}_{i_{1}...i_{p},v_{j_{1}...j_{p}}}(x,V)D^{j}_{\theta}V^{(c+l)}_{j_{1}...j_{p}}(s,x)ds
\nonumber\\
&&-\int_{\theta}^{u}\sum_{l=0}^{m}\sum_{j_{1}+...+j_{p}=c+l}{\cal
L}^{(c+1)}_{i_{1}...i_{p},\bar{v}_{j_{1}...j_{p}}}(x,V)D^{j}_{\theta}\bar{V}^{(c+l)}_{j_{1}...j_{p}}(s,x)ds
\nonumber\\
&&-\int_{\theta}^{u}\sum_{l=0}^{n}\sum_{j_{1}+...+j_{p}=c+l}{\cal
J}^{(c+1)}_{i_{1}...i_{p},v_{j_{1}...j_{p}}}(x,V)D^{j}_{\theta}V^{(c+l)}_{j_{1}...j_{p}}(s,x)dW(s)
\nonumber\\
&&+\int_{\theta}^{u}D^{j}_{\theta}\bar{V}^{(c)}_{i_{1}...i_{p}}(s,x)dW(s).
\nonumber
\end{eqnarray}
Thus, by taking $\theta=u$ in \eq{fdmbspde}, we know that the claims
in \eq{equaldis} are true. Hence, we complete the proof of
Lemma~\ref{masolution}. $\Box$
\end{proof}

\subsection{Second-Order Marlliavin Derivative Based
B-SPDE}\label{secondmad}

First, we use $\theta_{1}$ to replace $\theta$ in \eq{mbspde}.
Second, for each $j\in\{1,...,d\}$, $c\in\{0,1,...\}$, and
$(i_{1},...,i_{p})\in{\cal I}^{c}$, we define
\begin{eqnarray}
\bar{{\cal
L}}^{(c+1)}_{i_{1}...i_{p}}(x,V,D^{j}_{\theta_{1}}V)&=&\sum_{l=0}^{k}\sum_{j_{1}+...+j_{p}=c+l}{\cal
L}^{(c+1)}_{i_{1}...i_{p},v_{j_{1}...j_{p}}}(x,V)D^{j}_{\theta_{1}}V^{(c+l)}_{j_{1}...j_{p}}(s,x)
\elabel{semarI}\\
&&+\sum_{l=0}^{m}\sum_{j_{1}+...+j_{p}=c+l}{\cal
L}^{(c+1)}_{i_{1}...i_{p},\bar{v}_{j_{1}...j_{p}}}(x,V)D^{j}_{\theta_{1}}\bar{V}^{(c+l)}_{j_{1}...j_{p}}(s,x),
\nonumber\\
\bar{{\cal
J}}^{(c+1)}_{i_{1}...i_{p}}(x,V,D^{j}_{\theta_{1}}V)&=&\sum_{l=0}^{n}\sum_{j_{1}+...+j_{p}=c+l}{\cal
J}^{(c+1)}_{i_{1}...i_{p},v_{j_{1}...j_{p}}}(x,V)D^{j}_{\theta_{1}}V^{(c+l)}_{j_{1}...j_{p}}(s,x).
\elabel{semarII}
\end{eqnarray}
Then, we can obtain the following system of B-SPDEs for each
$\bar{j}\in\{1,...,d\}$ by taking Malliavin derivatives on both
sides of the equation in \eq{mbspde},
\begin{eqnarray}
&&D^{\bar{j}}_{\theta_{2}}D^{j}_{\theta_{1}}V^{(c)}_{i_{1}...i_{p}}(t,x)\elabel{mbspdeII}\\
&=&D^{\bar{j}}_{\theta_{2}}D^{j}_{\theta_{1}}H^{(c)}_{i_{1}...i_{p}}(x)\nonumber\\
&&+\int_{t}^{T}\sum_{l=0}^{k}\sum_{j_{1}+...+j_{p}=c+l}\bar{{\cal
L}}^{(c+2)}_{i_{1}...i_{p},v_{j_{1}...j_{p}}}(x,V,D^{j}_{\theta_{1}}V)
D^{\bar{j}}_{\theta_{2}}V^{(c+l)}_{j_{1}...j_{p}}(s,x)ds
\nonumber\\
&&+\int_{t}^{T}\sum_{l=0}^{k}\sum_{j_{1}+...+j_{p}=c+l}\bar{{\cal
L}}^{(c+2)}_{i_{1}...i_{p},(D^{j}_{\theta_{1}}v)_{j_{1}...j_{p}}}(x,V,D^{j}_{\theta_{1}}V)
D^{\bar{j}}_{\theta_{2}}D^{j}_{\theta_{1}}V^{(c+l)}_{j_{1}...j_{p}}(s,x)ds
\nonumber\\
&&+\int_{t}^{T}\sum_{l=0}^{m}\sum_{j_{1}+...+j_{p}=c+l}\bar{{\cal
L}}^{(c+2)}_{i_{1}...i_{p},(D^{j}_{\theta_{1}}\bar{v})_{j_{1}...j_{p}}}(x,V,D^{j}_{\theta_{1}}V)
D^{\bar{j}}_{\theta_{2}}\bar{V}^{(c+l)}_{j_{1}...j_{p}}(s,x)ds
\nonumber\\
&&+\int_{t}^{T}\sum_{l=0}^{m}\sum_{j_{1}+...+j_{p}=c+l}\bar{{\cal
L}}^{(c+2)}_{i_{1}...i_{p},(D^{j}_{\theta_{1}}\bar{v})_{j_{1}...j_{p}}}(x,V,D^{j}_{\theta_{1}}V)
D^{\bar{j}}_{\theta_{2}}D^{j}_{\theta_{1}}\bar{V}^{(c+l)}_{j_{1}...j_{p}}(s,x)ds
\nonumber\\
&&+\int_{t}^{T}\sum_{l=0}^{n}\sum_{j_{1}+...+j_{p}=c+l}\bar{{\cal
J}}^{(c+2)}_{i_{1}...i_{p},(D^{j}_{\theta_{1}}v)_{j_{1}...j_{p}}}(x,V,D^{j}_{\theta_{1}}V)
D^{\bar{j}}_{\theta_{2}}V^{(c+l)}_{j_{1}...j_{p}}(s,x)dW(s)
\nonumber\\
&&+\int_{t}^{T}\sum_{l=0}^{n}\sum_{j_{1}+...+j_{p}=c+l}\bar{{\cal
J}}^{(c+2)}_{i_{1}...i_{p},(D^{j}_{\theta_{1}}v)_{j_{1}...j_{p}}}(x,V,D^{j}_{\theta_{1}}V)
D^{\bar{j}}_{\theta_{2}}D^{j}_{\theta_{1}}V^{(c+l)}_{j_{1}...j_{p}}(s,x)dW(s)
\nonumber\\
&&-\int_{t}^{T}D^{\bar{j}}_{\theta_{2}}D^{j}_{\theta_{1}}\bar{V}^{(c)}_{i_{1}...i_{p}}(s,x)dW(s).
\nonumber
\end{eqnarray}
Furthermore, consider a properly chosen number sequence
$\gamma=\{\gamma_{c},c=0,1,2,...\}$ satisfying
$0<\gamma_{0}<\gamma_{1}<....$ such that the discussions for Theorem
2.1, Subsections~\ref{malliavinterm}-\ref{firstorderM}, and the
following Lemma~\ref{lemmasm} are meaningful, which can be
elaborated similar to the previous proof in
Subsection~\ref{malliavinterm}. Then, we can define the space
\begin{eqnarray}
&&\;\;\;{\cal O}^{D}_{\gamma}[0,T]=D^{2}_{{\cal
F}}([0,T],C^{\infty}(D,R^{q\times d\times d}))\times L^{2}_{{\cal
F},p}([0,T],C^{\infty}(D,R^{q\times d\times d\times d}))
\elabel{mmcombanach}
\end{eqnarray}
endowed with the norm similarly defined as in
\eq{comnorm}-\eq{kcomnorm}. Thus, we have the following lemma.
\begin{lemma}\label{lemmasm}
Under conditions as required in Theorem~\ref{errorIth} and with
Remark~\ref{firstremark}, if $(V(t,x),\bar{V}(t,x))$ $\in$ ${\cal
Q}_{{\cal F}}^{2}([0,T]\times D)$ is the adapted solution to
\eq{bspdef}, then,
\begin{eqnarray}
&(V(t,x),\bar{V}(t,x))\in
L^{\infty,2}_{2,2}([0,T]\times\Omega,(C^{\infty}(D,H))^{q\times
q\times d}). \nonumber
\end{eqnarray}
Furthermore, for $x\in D$, a version of the following
infinite-dimensional vector process
\begin{eqnarray}
\;\;\left\{(D_{\theta_{2}}D_{\theta_{1}}V^{(c)}_{i_{1}...i_{p}}(t,x),
D_{\theta_{2}}D_{\theta_{1}}\bar{V}^{(c)}_{i_{1}...i_{p}}(t,x)):
0\leq\theta_{1},\theta_{2},t\leq
T,c\in\{0,1,...\},\;\;(i_{1},...,i_{p})\in{\cal I}^{c}\right\}
\nonumber
\end{eqnarray}
is given by the system in \eq{mbspdeII}. In addition, for $0\leq
t<\theta_{1}\wedge\theta_{2}\leq T$ and $1\leq\bar{j},j\leq d$,
\begin{eqnarray}
&D^{\bar{j}}_{\theta_{2}}D^{j}_{\theta_{1}}V^{(c)}_{i_{1}...i_{p}}(t,x)=0,
\;\;\;D^{\bar{j}}_{\theta_{2}}D^{j}_{\theta_{1}}\bar{V}^{(c)}_{i_{1}...i_{p}}(t,x)=0,
\elabel{maequalzeroI}
\end{eqnarray}
and
\begin{eqnarray}
&&\left\{D_{t}\bar{V}^{(c)}_{i_{1}...i_{p}}(t,x),t\in[0,T],c\in\{0,1,...\},(i_{1},...,i_{p})\in{\cal
I}^{c},
x\in D\right\}\elabel{equaldisII}\\
&&=^{d}\{D_{t}D_{t}V^{(c)}_{i_{1}...i_{p}}(t,x)+\sum_{l=0}^{n}\sum_{j_{1}+...+j_{p}=c+l}{\cal
J}^{(c+1)}_{i_{1}...i_{p},v_{j_{1}...j_{p}}}(x,V)D_{t}V^{(c+l)}_{j_{1}...j_{p}}(t,x),
\nonumber\\
&&\;\;\;\;\;\;\;\;t\in[0,T],c\in\{0,1,...\},(i_{1},...,i_{p})\in{\cal
I}^{c},x\in D\}. \nonumber
\end{eqnarray}
\end{lemma}
\begin{proof}
Let
\begin{eqnarray}
L(t)&\equiv&\|V(t)\|_{C^{\infty}(D,q)}
+\left\|\bar{V}(t)\right\|_{C^{\infty}(D,qd)}\nonumber\\
&&+\left\|D_{\theta_{1}}V(t)\right\|_{C^{\infty}(D,qd)}
+\left\|D_{\theta_{1}}\bar{V}(t)\right\|_{C^{\infty}(D,qdd)}
\nonumber\\
&&+\left\|D_{\theta_{2}}V(t)\right\|_{C^{\infty}(D,qd)}
+\left\|D_{\theta_{2}}\bar{V}(t)\right\|_{C^{\infty}(D,qdd)}.
\nonumber
\end{eqnarray}
Then, similar to \eq{stoptime}, we define a sequence of
nondecreasing $\{{\cal F}_{t}\}$-stopping times along
$w\in\{0,1,...\}$ as follows,
\begin{eqnarray}
&&\tau_{w}\equiv T\wedge\inf\left\{t>0,L(t)>w\right\},
\elabel{stoptimeI}
\end{eqnarray}
which satisfies $\tau_{w}\rightarrow T$ a.s. as
$w\rightarrow\infty$. Thus, by the similar arguments as used in the
proofs of Lemmas~\ref{limitadapted}-\ref{masolution}, we can provide
a proof for Lemma~\ref{lemmasm}. $\Box$
\end{proof}

\subsection{Priori Estimates}\label{prioriI}

\begin{lemma}\label{prioriest}
Under conditions as required in Theorem~\ref{errorIth} and with
Remark~\ref{firstremark}, if $(V^{i}(t,x)$, $\bar{V}^{i}(t,x))$ for
each $i\in\{1,2\}$ is the unique adapted solution to equation
\eq{bspdef} with terminal condition $H^{i}(x)$, then,
\begin{eqnarray}
\left\|(V^{i},\bar{V}^{i})\right\|^{2}_{{\cal
M}_{\gamma}^{D}[0,T]}&\leq&
\bar{C}\left(1+\left\|H^{i}\right\|^{2}_{L^{2}_{{\cal
F}_{T}}(\Omega,C^{\infty}(D,R^{q}))}\right),
\elabel{vbarvine}\\
\left\|(V^{2},\bar{V}^{2})-(V^{1},\bar{V}^{1})\right\|^{2}_{{\cal
M}_{\gamma}^{D}[s,t]}&\leq&
\bar{C}\left\|H^{2}-H^{1}\right\|^{2}_{L^{2}_{{\cal
F}_{T}}(\Omega,C^{\infty}(C,R^{q}))} \elabel{vbarvineII}
\end{eqnarray}
for some nonnegative constant $\bar{C}$ only depending on the
terminal time $T$, the region $D$. Furthermore, for each
$c\in\{0,1,...\}$ and any $s,t\in[0,T]$ with $s\leq t$, we have
\begin{eqnarray}
E\left[\left\|V^{i}(t)-V^{i}(s)\right\|^{2}_{C^{c}(D,q)}\right]
&\leq& C(t-s), \elabel{vbarvineI}\\
E\left[\left\|\bar{V}^{i}(t)-\bar{V}^{i}(s)\right\|^{2}_{C^{c}(D,qd)}\right]
&\leq&C(t-s).\elabel{barvbarvineI}
\end{eqnarray}
for some nonnegative constant $C$ only depending on the terminal
time $T$, the region $D$, and the terminal random variable.
\end{lemma}
\begin{proof}
By applying the It$\hat{o}$'s formula and the similar proof as used
for \eq{ftheta}, we know that the claims in
\eq{vbarvine}-\eq{vbarvineII} are true. Now, consider the B-SPDE
\eq{bspdef} over $[s,t]$ with terminal condition $V(t,x)$. Then, by
\eq{blipschitzoI}-\eq{blipschitzI}, \eq{equaldis}, \eq{conpi} and
\eq{vvbar}, we know that
\begin{eqnarray}
&&E\left[\left\|V^{i}(t)-V^{i}(s)\right\|^{2}_{C^{c}(D,q)}\right]
\elabel{dvproof}\\
&\leq&C_{1}\int_{s}^{t}\left(1+\left\|V^{i}(r)\right\|^{2}_{C^{k+c}(D,q)}
+\left\|D_{r}V^{i}(r)\right\|^{2}_{C^{k+c}(D,qd)}\right)dr
\nonumber\\
&\leq&C_{2}(t-s)\left(\left\|(V^{i},\bar{V}^{i})\right\|^{2}_{{\cal
M}_{\gamma}^{D}[0,T]}+\left\|(DV^{i},D\bar{V}^{i})\right\|^{2}_{{\cal
N}_{\gamma}^{D}[0,T]}\right)
\nonumber\\
&\leq&C_{3}(t-s), \nonumber
\end{eqnarray}
where $C_{1}$ and $C_{2}$ are some nonnegative constants.

Furthermore, by using \eq{equaldis}, \eq{equaldisII} in
Lemma~\ref{lemmasm}, and the similar argument as used in
\eq{dvproof}, we know that
\begin{eqnarray}
&&E\left[\left\|\bar{V}^{i}(t)-\bar{V}^{i}(s)\right\|^{2}_{C^{c}(D,qd)}\right]
\nonumber\\
&\leq&C_{3}E\left[\left\|V^{i}(t)-V^{i}(s)\right\|^{2}_{C^{c}(D,q)}
+\left\|D_{t}V^{i}(t)-D_{s}V^{i}(s)\right\|^{2}_{C^{c}(D,qd)}\right]
\nonumber\\
&\leq&C_{4}\int_{s}^{t}\left(1+\left\|V^{i}(r)\right\|^{2}_{C^{k+c}(D,q)}
+\left\|D_{r}V^{i}(r)\right\|^{2}_{C^{k+c}(D,qd)}
+\left\|D_{r}D_{r}V^{i}(r)\right\|^{2}_{C^{k+c}(D,qdd)}\right)dr
\nonumber\\
&\leq&C_{5}(t-s)\left(\left\|(V^{i},\bar{V}^{i})\right\|^{2}_{{\cal
M}_{\gamma}^{D}[0,T]}+\left\|(DV^{i},D\bar{V}^{i})\right\|^{2}_{{\cal
N}_{\gamma}^{D}[0,T]}+\left\|(DDV^{i},DD\bar{V}^{i})\right\|^{2}_{{\cal
O}_{\gamma}^{D}[0,T]}\right)
\nonumber\\
&\leq&C_{6}(t-s), \nonumber
\end{eqnarray}
where $C_{4}$-$C_{6}$ are some nonnegative constants. Finally, take
$C=\max\{C_{3},C_{6}\}$ sure that both \eq{vbarvineI} and
\eq{barvbarvineI} are true. Hence, we complete the proof of
Lemma~\ref{prioriest}. $\Box$
\end{proof}


\subsection{Representation Formulas}\label{repres}

Concerning Algorithm~\ref{algorithmI}, we first define the following
quantities as $j_{0}$ decreases from $n_{0}$ to 1 for each
$c\in\{0,1,...,M\}$ and $x\in D$,
\begin{eqnarray}
V^{(c)}_{i_{1}...i_{p},\pi,0}(t_{n_{0}},x)&\equiv&H^{(c)}_{i_{1}...i_{p}}(x),
\;\;
\bar{V}^{(c)}_{i_{1}...i_{p},\pi,0}(t_{n_{0}},x)=0,\elabel{cIalgorithmI}\\
\;\;\;\;\;V^{(c)}_{i_{1}...i_{p},\pi,0}(t_{j_{0}-1},x)
&\equiv&E\left[\left.V^{(c)}_{i_{1}...i_{p}}(t_{j_{0}},x) +{\cal
L}^{(c)}_{i_{1}...i_{p}}(t_{j_{0}},x,V(t_{j_{0}},x))
\Delta^{\pi}_{j_{0}}\right|{\cal F}_{t_{j_{0}-1}}\right],
\elabel{cIIalgorithmI}\\
\;\;\;\;\;\bar{V}^{(c)}_{i_{1}...i_{p},\pi,0}(t_{j_{0}-1})
&\equiv&\frac{1}{\Delta^{\pi}_{j_{0}}}
E\left[\left.V^{(c)}_{i_{1}...i_{p}}(t_{j_{0}},x)\Delta^{\pi}W_{j_{0}}\right|{\cal
F}_{t_{j_{0}-1}}\right] \elabel{cIIIalgorithmI}\\
&&+E\left[\left.{\cal
L}^{(c)}_{i_{1}...i_{p}}(t_{j_{0}},x,V(t_{j_{0}},x))
\Delta^{\pi}W_{j_{0}}\right|{\cal
F}_{t_{j_{0}-1}}\right]\nonumber\\
&&+{\cal
J}^{(c)}_{i_{1}...i_{p}}(t_{j_{0}-1},x,V_{\pi}(t_{j_{0}-1},x)).
\nonumber
\end{eqnarray}
Then, we consider the following iterative procedure,
\begin{eqnarray}
V_{i_{1}...i_{p},\pi,1}^{(c)}(t,x)&=&H^{(c)}_{i_{1}...i_{p},\pi,1}(t_{j_{0}},x)
\elabel{IIalgorithmIm}\\
&&+\int_{t}^{t_{j_{0}}}\left({\cal
J}_{i_{1}...i_{p}}^{(c)}(s,x,V_{\pi,1}(s,x))
-\bar{V}^{(c)}_{i_{1}...i_{p},\pi,1}(s,x)\right)dW(s) \nonumber
\end{eqnarray}
for each $t\in[t_{j_{0}-1},t_{j_{0}})$ and $x\in D$, where
\begin{eqnarray}
\;\;H^{(c)}_{i_{1}...i_{p},\pi,1}(t_{j_{0}},x)&=&V_{i_{1}...i_{p}}(t_{j_{0}},x)
+{\cal L}^{(c)}_{i_{1}...i_{p}}(t_{j_{0}},x,V(t_{j_{0}},x))
\Delta^{\pi}_{j_{0}}, \elabel{initialHJO}\\
H^{(c)}_{i_{1}...i_{p},\pi,1}(t_{n_{0}},x)&=&H^{(c)}_{i_{1}...i_{p}}(x).
\elabel{initialHJOI}
\end{eqnarray}
Note that the equation displayed in \eq{IIalgorithmIm} for each
$j_{0}\in\{n_{0},n_{0}-1,...,1\}$ is a B-SPDE with terminal value
$H^{(c)}_{i_{1},...,i_{p},\pi,1}(t_{j_{0}},x)$. Then, we have the
following lemma.
\begin{lemma}\label{clarklemma}
Under conditions \eq{termI} and \eq{blipschitz}-\eq{blipschitzI},
there is a unique adapted and square-integrable solution
$(V^{(c)}_{i_{1}...i_{p},\pi,1}(t,x),\bar{V}^{(c)}_{i_{1}...i_{p},\pi,1}(t,x))$
to the B-SPDE in \eq{IIalgorithmIm} over
$t\in[t_{j_{0}-1},t_{j_{0}})$ and $x\in D$ for each
$c\in\{0,1,...,M\}$ and $(i_{1},...,i_{p})\in{\cal I}^{c}$.
Moreover, we have
\begin{eqnarray}
&&V^{(c)}_{i_{1}...i_{p},\pi,0}(t,x)=V^{(c)}_{i_{1}...i_{p},\pi,1}(t,x),\;\;
t\in[t_{j_{0}-1},t_{j_{0}}),\;x\in D,
\elabel{equalI}\\
&&\bar{V}^{(c)}_{i_{1}...i_{p},\pi,0}(t_{j_{0}-1},x)
\elabel{equalII}\\
&=&\frac{1}{\Delta^{\pi}_{j_{0}}}
E\left[\left.\int_{t_{j_{0}-1}}^{t_{j_{0}}}\left(\bar{V}^{(c)}_{i_{1}...i_{p},\pi,1}(s,x)
-{\cal J}^{(c)}_{i_{1}...i_{p}}(s,x,V_{\pi,1}(s,x))\right)ds
\right|{\cal F}_{t_{j_{0}-1}}\right]\nonumber\\
&&+{\cal
J}^{(c)}_{i_{1}...i_{p}}(t_{j_{0}-1},x,V_{\pi,1}(t_{j_{0}-1},x)).
\nonumber
\end{eqnarray}
\end{lemma}
\begin{proof}
Without loss of generality, we only consider the case that $c=0$.
Owing to conditions \eq{termI}, \eq{blipschitz}-\eq{blipschitzI},
and Theorem 2.1, we know that there is a unique adapted and
mean-square integrable solution
$(V_{\pi,1}(t,x),\bar{V}_{\pi,1}(t,x))$ to the B-SPDE in
\eq{IIalgorithmIm} over $t\in[t_{j_{0}-1},t_{j_{0}})$ and $x\in D$
for each $j_{0}\in\{n_{0},n_{0}-1,...,1\}$. Then, by taking the
conditional expectations on both sides of \eq{IIalgorithmIm} at each
time $t$ and the independent increment property of the Brownian
motion, we know that the claim in \eq{equalI} is true by a backward
induction method in terms of $j_{0}=n_{0},n_{0}-1,...,1$.

Now, it follows from Lemma~\ref{masolution} that
$H_{\pi,1}(t_{j_{0}},x)\in{\cal D}^{1,2}_{\infty}\bigcap
L^{2}_{{\cal F}_{t_{j_{0}}}}(\Omega,C^{\infty}(D,R^{q}))$ for each
$j_{0}\in\{n_{0},n_{0}-1,...,1\}$. Then, for each $c\in\{0,1,...\}$,
$x\in D$, and $(i_{1},...,i_{p})\in{\cal I}^{c}$, it follows from
Lemma~\ref{chocone} that
\begin{eqnarray}
&&H_{i_{1}...i_{p},\pi,1}^{(c)}(t_{j_{0}},x)
=E\left[H_{i_{1}...i_{p},\pi,1}^{(c)}(t_{j_{0}},x)\right]
+\int_{0}^{t_{j_{o}}}E\left[\left.D_{t}H_{i_{1}...i_{p},\pi,1}^{(c)}(t_{j_{0}},x)\right|{\cal
F}_{t}\right]dW(t) \elabel{clarkoconeI}
\end{eqnarray}
with
\begin{eqnarray}
E\left[\left|E\left[\left.D_{t}H_{i_{1}...i_{p},\pi,1}^{(c)}(t_{j_{0}},x)\right|{\cal
F}_{t}\right]\right|^{2}dt\right]<\infty.\nonumber
\end{eqnarray}
Thus, we know that
\begin{eqnarray}
E\left[H_{i_{1}...i_{p},\pi,1}^{(c)}(t_{j_{0}},x)\right]
&=&E\left[\left.H_{i_{1}...i_{p},\pi,1}^{(c)}(t_{j_{0}},x)\right|{\cal
F}_{t_{j_{0}-1}}\right] \elabel{clarkoconeII}\\
&&-\int_{0}^{t_{j_{0}-1}}
E\left[\left.D_{s}H_{i_{1}...i_{p},\pi,1}^{(c)}(t_{j_{0}},x)\right|{\cal
F}_{s}\right]dW(s). \nonumber
\end{eqnarray}
Hence, it follows from \eq{clarkoconeI} and \eq{clarkoconeII} that
\begin{eqnarray}
H_{i_{1}...i_{p},\pi,1}^{(c)}(t_{j_{0}},x)
&=&E\left[\left.H_{i_{1}...i_{p},\pi,1}^{(c)}(t_{j_{0}},x)\right|{\cal
F}_{t_{j_{0}-1}}\right]\elabel{clarkoconeIII}\\
&&+\int_{t_{j_{0}-1}}^{t_{j_{0}}}
E\left[\left.D_{s}H_{i_{1}...i_{p},\pi,1}^{(c)}(t_{j_{0}},x)\right|{\cal
F}_{s}\right]dW(s). \nonumber
\end{eqnarray}
Now, for any set $A\in{\cal F}_{t_{j_{0}-1}}$, we have
\begin{eqnarray}
E\left[H_{i_{1}...i_{p},\pi,1}^{(c)}(t_{j_{0}},x)\Delta^{\pi}W_{j_{0}}I_{A}
\right]
&=&E\left[H_{i_{1}...i_{p},\pi,1}^{(c)}(t_{j_{0}},x)\int_{t_{j_{0}-1}}^{t_{j_{0}}}
I_{A}dW(s)\right]\elabel{calII}\\
&=&E\left[\int_{t_{j_{0}-1}}^{t_{j_{0}}}I_{A}D_{s}H_{i_{1}...i_{p},\pi,1}^{(c)}(t_{j_{0}},x)ds\right]
\nonumber\\
&=&E\left[\int_{t_{j_{0}-1}}^{t_{j_{0}}}I_{A}E\left[\left.D_{s}H_{i_{1}...i_{p},\pi,1}^{(c)}(t_{j_{0}},x)
\right|{\cal F}_{s}\right]ds\right]
\nonumber\\
&=&E\left[I_{A}\int_{t_{j_{0}-1}}^{t_{j_{0}}}
E\left.\left[D_{s}H_{i_{1}...i_{p},\pi,1}^{(c)}(t_{j_{0}},x)\right|{\cal
F}_{s}\right]ds\right],\nonumber
\end{eqnarray}
where the second equality is obtained from the Malliavin integration
by parts formula (see, e.g., Theorem A.3.9 in page 283 of Biagini
{\em et al.}~\cite{biahu:stocal}) and the third equality follows
from the tower property for conditional expectations owing to the
square integrability and the Fubini's theorem. Hence, it follows
from the definition of conditional expectation that
\begin{eqnarray}
&&\;\;E\left[\left.H_{i_{1}...i_{p},\pi,1}^{(c)}(t_{j_{0}},x)\Delta^{\pi}W_{j_{0}}
\right|{\cal
F}_{t_{j_{0}-1}}\right]=E\left[\left.\int_{t_{j_{0}-1}}^{t_{j_{0}}}
E\left.\left[D_{s}H_{i_{1}...i_{p},\pi,1}^{(c)}(t_{j_{0}},x)\right|{\cal
F}_{s}\right]ds\right|{\cal F}_{t_{j_{0}-1}}\right].\elabel{calI}
\end{eqnarray}
Therefore we have
\begin{eqnarray}
&&\left|E\left[\left.\int_{t_{j_{0}-1}}^{t_{j_{0}}}
\left(E\left[\left.D_{s}H_{i_{1}...i_{p},\pi,1}^{(c)}(t_{j_{0}},x)\right|{\cal
F}_{s}\right]\right.\right.\right.\right.\elabel{fmalliavin}\\
&&\;\;\;\;\;\;\;\;\;\;\;\;\;\;\;\;\;
\left.\left.\left.\left.-\left(\bar{V}^{(c)}_{i_{1}...i_{p},\pi,1}(s,x)-{\cal
J}^{(c)}_{i_{1}...i_{p}}(s,x,V_{\pi,1}(s,x))\right)\right)ds\right|{\cal
F}_{t_{j_{0}-1}}\right]\right|\nonumber\\
&\leq&\left(E\left[\left.\int_{t_{j_{0}-1}}^{t_{j_{0}}}
\left(E\left[\left.D_{s}H_{i_{1}...i_{p},\pi,1}^{(c)}(t_{j_{0}},x)\right|{\cal
F}_{s}\right]\right.\right.\right.\right.\nonumber\\
&&\;\;\;\;\;\;\;\;\;\;\;\;\;\;\;\;\;
\left.\left.\left.\left.-\left(\bar{V}^{(c)}_{i_{1}...i_{p},\pi,1}(s,x)-{\cal
J}^{(c)}_{i_{1}...i_{p}}(s,x,V_{\pi,1}(s,x))\right)\right)^{2}ds\right|{\cal
F}_{t_{j_{0}-1}}\right]\right)^{1/2}\nonumber\\
&=&\left(E\left[\left.\left(H_{i_{1}...i_{p},\pi,1}^{(c)}(t_{j_{0}},x)
-E\left[\left.H_{i_{1}...i_{p},\pi,1}^{(c)}(t_{j_{0}},x)\right|{\cal
F}_{t_{j_{0}-1}}\right]\right.\right.\right.\right.\nonumber\\
&&\;\;\;\;\;\;\;\;\;
\left.\left.\left.\left.-\left(H^{(c)}_{i_{1}...i_{p},\pi,1}(t_{j_{0}},x)
-V^{(c)}_{i_{1}...i_{p},\pi,1}(t_{j_{0}-1},x)\right)\right)^{2}
\right|{\cal F}_{t_{j_{0}-1}}\right]\right )^{1/2} \nonumber\\
&=& 0,\nonumber
\end{eqnarray}
where the first inequality in \eq{fmalliavin} follows from the
H\"{o}lder's and Jensen's inequalities; the first equality in
\eq{fmalliavin} follows from \eq{IIalgorithmIm}, \eq{clarkoconeIII},
and the It$\hat{o}$'s isometry; the second inequality in
\eq{fmalliavin} follows from \eq{IIalgorithmIm} and \eq{equalI}.
Thus, it follows from \eq{calI}-\eq{fmalliavin},
\eq{cIIIalgorithmI}, and \eq{equalI} that \eq{equalII} is true.
Hence, we finish the proof of Lemma~\ref{clarklemma}. $\Box$
\end{proof}

\subsection{Proof of Theorem~\ref{errorIth}}\label{proofcon}

First, we note that the convention given in Remark~\ref{firstremark}
will be employed in the following proof. Then, for each
$t\in[t_{j_{0}-1},t_{j_{0}})$, $x\in{\cal X}$, and
$c\in\{0,1,...,M\}$, we can obtain that
\begin{eqnarray}
E\left[\left\|\Delta V^{(c)}(t,x)\right\|^{2}\right] &\leq&
5E\left[\left\|V^{(c)}(t,x)-V^{(c)}(t_{j_{0}-1},x)\right\|^{2}\right]
\elabel{vestimate}\\
&&+5E\left[\left\|V^{(c)}(t_{j_{0}-1},x)-V^{(c)}_{\pi,0}(t_{j_{0}-1},x)\right\|^{2}\right]
\nonumber\\
&&+5E\left[\left\|V^{(c)}_{\pi,0}(t_{j_{0}-1},x)-V^{(c)}_{\pi,1}(t_{j_{0}-1},x)\right\|^{2}\right]
\nonumber\\
&&+5E\left[\left\|V^{(c)}_{\pi,1}(t_{j_{0}-1},x)-V^{(c)}_{\pi,1}(t,x)\right\|^{2}\right]
\nonumber\\
&&+5E\left[\left\|V^{(c)}_{\pi,1}(t,x)-V^{(c)}_{\pi}(t,x)\right\|^{2}\right],
\nonumber
\end{eqnarray}
which implies that there is some nonnegative constant $K_{0}$ such
that
\begin{eqnarray}
&&E\left[\left\|\Delta V^{(c)}(t,x)\right\|^{2}\right] \leq
K_{0}\pi.\elabel{firstclaim}
\end{eqnarray}

In fact, for the first and fourth terms on the right-hand side of
\eq{vestimate}, it follows from \eq{vbarvineI} in
Lemma~\ref{prioriest} that there is some nonnegative constant
$K_{1}$ such that
\begin{eqnarray}
&&E\left[\left\|V^{(c)}(t,x)-V^{(c)}(t_{j_{0}-1},x)\right\|^{2}\right]\leq
K_{1}\pi, \elabel{firstterm}\\
&&E\left[\left\|V^{(c)}_{\pi,1}(t_{j_{0}-1},x)-V^{(c)}_{\pi,1}(t,x)\right\|^{2}\right]\leq
K_{1}\pi. \elabel{fourthterm}
\end{eqnarray}

For the third term on the right-hand side of \eq{vestimate}, it
follows from \eq{equalI} in Lemma~\ref{clarklemma} that
\begin{eqnarray}
&&E\left[\left\|V^{(c)}_{\pi,0}(t_{j_{0}-1},x)-V^{(c)}_{\pi,1}(t_{j_{0}-1},x)\right\|^{2}\right]=0\leq
K_{1}\pi.\elabel{thirdterm}
\end{eqnarray}

For the second term on the right-hand side of \eq{vestimate}, it
follows from \eq{cIalgorithmI}-\eq{cIIalgorithmI},
\eq{blipschitz}-\eq{blipschitzo}, the Jensen's inequality, and
\eq{vbarvineI}-\eq{barvbarvineI} that
\begin{eqnarray}
&&E\left[\left\|V^{(c)}(t_{j_{0}-1},x)-V^{(c)}_{\pi,0}(t_{j_{0}-1},x)\right\|^{2}\right]
\elabel{secondterm}\\
&\leq&\bar{K_{2}}\int_{t_{j_{0}-1}}^{t_{j_{0}}}\left(\left\|V(s)-V(t_{j_{0}})\right\|^{2}_{C^{k+c}(D,q)}
+\left\|\bar{V}(s)-\bar{V}(t_{j_{0}})\right\|^{2}_{C^{k+c}(D,qd)}\right)ds \nonumber\\
&\leq&K_{2}\pi,\nonumber
\end{eqnarray}
where $\bar{K}_{2}$ and $K_{2}$ are some nonnegative constants.

For the last term on the right-hand side of \eq{vestimate}, it
follows from \eq{equalI} in Lemma~\ref{clarklemma},
Lemma~\ref{prioriest}, and Taylor's Theorem that
\begin{eqnarray}
&&E\left[\left\|V^{(c)}_{\pi,1}(t,x)-V^{(c)}_{\pi}(t,x)\right\|^{2}\right]
\elabel{fifthterm}\\
&\leq&E\left[\left\|V^{(c)}_{\pi,0}(t,x)-V^{(c)}(t,x)\right\|^{2}\right]
+E\left[\left\|V^{(c)}(t,x)-V^{(c)}(t,\xi(x))\right\|^{2}\right]
\nonumber\\
&\leq&\bar{K_{2}}\int_{t}^{t_{j_{0}}}
\left(\left\|V(s)-V(t_{j_{0}})\right\|^{2}_{C^{k+c}(D,q)}
+\left\|\bar{V}(s)-\bar{V}(t_{j_{0}})\right\|^{2}_{C^{k+c}(D,qd)}\right)ds
\nonumber\\
&&+\bar{K}_{3}\pi E\left[\left\|V^{(c+1)}(t,\xi_{1}(x))\right\|^{2}\right]
\nonumber\\
&\leq&K_{3}\pi,
\nonumber
\end{eqnarray}
where $\bar{K}_{3}$ and $K_{3}$ are some nonnegative constants,
$\xi(x)$ and $\xi_{1}(x)$ along each sample path are in some small
neighborhoods centered at $x$. Therefore, it follows from
\eq{firstterm}-\eq{fifthterm} that the claim in \eq{firstclaim} is
true.

Furthermore, for each $t\in[t_{j_{0}-1},t_{j_{0}})$, $x\in{\cal X}$,
and $c\in\{0,1,...,M\}$, we have that
\begin{eqnarray}
E\left[\left\|\Delta\bar{V}^{(c)}(t,x)\right\|^{2}\right]
&\leq&5E\left[\left\|\bar{V}^{(c)}(t,x)-\bar{V}^{(c)}(t_{j_{0}-1},x)\right\|^{2}\right]
\elabel{secondest}\\
&&+5E\left[\left\|\bar{V}^{(c)}(t_{j_{0}-1},x)-\bar{V}^{(c)}_{\pi,0}(t_{j_{0}-1},x)\right\|^{2}
\right]\nonumber\\
&&+5E\left[\left\|\bar{V}^{(c)}_{\pi,0}(t_{j_{0}-1},x)-\bar{V}^{(c)}_{\pi,1}(t_{j_{0}-1},x)\right\|^{2}
\right]\nonumber\\
&&+5E\left[\left\|\bar{V}^{(c)}_{\pi,1}(t_{j_{0}-1},x)-\bar{V}^{(c)}_{\pi,1}(t,x)\right\|^{2}
\right]\nonumber\\
&&+5E\left[\left\|\bar{V}^{(c)}_{\pi,1}(t,x)-\bar{V}^{(c)}_{\pi}(t,x)\right\|^{2}\right],
\nonumber
\end{eqnarray}
which implies that there is some nonnegative constant $\bar{K}_{0}$
such that
\begin{eqnarray}
&&E\left[\left\|\Delta\bar{V}^{(c)}(t,x)\right\|^{2}\right] \leq
\bar{K}_{0}\pi.\elabel{secondclaim}
\end{eqnarray}

In fact, for the first and fourth terms on the right-hand side of
\eq{secondest}, it follows from \eq{barvbarvineI} in
Lemma~\ref{prioriest} that there is some nonnegative constant
$\kappa_{1}$ such that
\begin{eqnarray}
&&E\left[\left\|\bar{V}^{(c)}(t,x)-\bar{V}^{(c)}(t_{j_{0}-1},x)\right\|^{2}\right]
\leq\kappa_{1}\pi,
\elabel{sfirstterm}\\
&&E\left[\left\|\bar{V}^{(c)}_{\pi,1}(t_{j_{0}-1},x)
-\bar{V}^{(c)}_{\pi,1}(t,x)\right\|^{2}\right]\leq\kappa_{1}\pi.
\elabel{sfourthterm}
\end{eqnarray}

For the third term on the right-hand side of \eq{secondest}, note
that
\begin{eqnarray}
&&E\left[\left(X-E\left[X\left.|{\cal
F}_{t_{j_{0}-1}}\right.\right]\right)^{2}\right]\leq
E\left[\left(X-Y\right)^{2}\right]
\elabel{simine}
\end{eqnarray}
for any two $L^{2}_{{\cal F}_{t_{j_{0}}}}(P)$-integrable random
variables $X$ and $Y$. Then, it follows from \eq{equalII} in
Lemma~\ref{clarklemma}, the H$\ddot{o}$lder's inequality,
\eq{simine}, \eq{sfourthterm}, \eq{blipschitzo}, and \eq{fourthterm}
that
\begin{eqnarray}
&&E\left[\left\|\bar{V}^{(c)}_{\pi,0}(t_{j_{0}-1},x)
-\bar{V}^{(c)}_{\pi,1}(t_{j_{0}-1},x)\right\|^{2}\right]
\elabel{sthirdterm}\\
&\leq&\frac{2}{\Delta^{\pi}_{j_{0}}}
E\left[\int_{t_{j_{0}-1}}^{t_{j_{0}}}
\left\|\left(\bar{V}^{(c)}_{\pi,1}(t_{j_{0}-1},x)-{\cal
J}^{(c)}(x,V_{\pi,1}(t_{j_{0}-1},x)\right)\right.\right.
\nonumber\\
&&\;\;\;\;\;\;\;\;\;\;\;\;\;\;\;\;\;\;\;\;\;\;\;
\left.\left.-\left(\bar{V}^{(c)}_{\pi,1}(s,x)-{\cal
J}^{(c)}(x,V_{\pi,1}(s,x)\right)\right\|^{2}ds\right]
\nonumber\\
&&+\frac{2}{\Delta^{\pi}_{j_{0}}}
E\left[\int_{t_{j_{0}-1}}^{t_{j_{0}}}
\left\|\left(\bar{V}^{(c)}_{\pi,1}(s,x)-{\cal
J}^{(c)}(x,V_{\pi,1}(s,x)\right)\right.\right.
\nonumber\\
&&\;\;\;\;\;\;\;\;\;\;\;\;\;\;\;\;\;\;\;\;\;\;\;\;
\left.\left.-E\left[\left.\int_{t_{j_{0}-1}}^{t_{j_{0}}}\left(\bar{V}^{(c)}_{\pi,1}(s,x)-{\cal
J}^{(c)}(x,V_{\pi,1}(s,x)\right)\right|{\cal
F}_{t_{j_{0}-1}}\right]\right\|^{2}ds\right]
\nonumber\\
&\leq&\frac{4}{\Delta^{\pi}_{j_{0}}}E\left[\int_{t_{j_{0}-1}}^{t_{j_{0}}}
\left\|\left(\bar{V}^{(c)}_{\pi,1}(t_{j_{0}-1},x)-{\cal
J}^{(c)}(x,V_{\pi,1}(t_{j_{0}-1},x)\right)\right.\right.
\nonumber\\
&&\;\;\;\;\;\;\;\;\;\;\;\;\;\;\;\;\;\;\;\;\;
\left.\left.-\left(\bar{V}^{(c)}_{\pi,1}(s,x)-{\cal
J}^{(c)}(x,V_{\pi,1}(s,x)\right)\right\|^{2}ds\right]
\nonumber\\
&\leq&\kappa_{1}\pi
\nonumber
\end{eqnarray}
for some nonnegative constant $\kappa_{1}$.

For the second term on the right-hand side of \eq{secondest}, it
follows from \eq{sthirdterm}, the special form of the terminal
variable in \eq{IIalgorithmIm} ,
Lemmas~\ref{lemmasm}-\ref{prioriest}, and the proof of the first
three terms in \eq{vestimate} that
\begin{eqnarray}
&&E\left[\left\|\bar{V}^{(c)}(t_{j_{0}-1},x)-\bar{V}^{(c)}_{\pi,0}(t_{j_{0}-1},x)\right\|^{2}\right]
\elabel{ssecondterm}\\
&\leq&2\left(E\left[\left\|\bar{V}^{(c)}(t_{j_{0}-1},x)-\bar{V}^{(c)}_{\pi,1}(t_{j_{0}-1},x)\right\|^{2}\right]
+E\left[\left\|\bar{V}^{(c)}_{\pi,1}(t_{j_{0}-1},x)-\bar{V}^{(c)}_{\pi,0}(t_{j_{0}-1},x)\right\|^{2}\right]\right)
\nonumber\\
&\leq&2\int_{t_{j_{0}-1}}^{t_{j_{0}}}E\left[\left\|\bar{V}^{(c)}(s,x)
-\bar{V}^{(c)}_{\pi,1}(s,x)\right\|^{2}\right]ds
+\bar{\kappa}_{2}\int_{t_{j_{0}-1}}^{t_{j_{0}}}E\left[\left\|V^{(c)}(s,x)-
V^{(c)}_{\pi,1}(s,x)\right\|^{2}\right]ds
\nonumber\\
&&+\bar{\kappa}_{2}\pi
\nonumber\\
&\leq&\kappa_{2}\pi,\nonumber
\end{eqnarray}
where $\kappa_{2}$ and $\bar{\kappa}_{2}$ are some nonnegative
constants.

For the last term on the right-hand side of \eq{secondest}, it
follows from Lemma~\ref{prioriest}, Taylor's Theorem, and the proof
in \eq{ssecondterm} that
\begin{eqnarray}
&&E\left[\left\|\bar{V}^{(c)}_{\pi,1}(t,x)-\bar{V}^{(c)}_{\pi}(t,x)\right\|^{2}\right]
\elabel{sfifthterm}\\
&\leq&E\left[\left\|\bar{V}^{(c)}_{\pi,1}(t,x)-\bar{V}^{(c)}(t,x)\right\|^{2}\right]
+E\left[\left\|\bar{V}^{(c)}(t,x)-\bar{V}^{(c)}(t,\xi(x))\right\|^{2}\right]
\nonumber\\
&\leq&\bar{\kappa}_{2}\pi+\bar{\kappa}_{3}\pi E\left[\left\|V^{(c+1)}(t,\xi_{1}(x))\right\|^{2}\right]
\nonumber\\
&\leq&\kappa_{3}\pi,\nonumber
\end{eqnarray}
where $\bar{\kappa}_{3}$ and $\kappa_{3}$ are some nonnegative
constants, $\xi(x)$ and $\xi_{1}(x)$ along each sample path are in
some small neighborhoods centered at $x$. Therefore, it follows from
\eq{sfirstterm}-\eq{sfifthterm} that the claim in \eq{secondclaim}
is true.

Finally, the claim in \eq{errorest} for Algorithm~\ref{algorithmI}
follows from \eq{firstclaim} and \eq{secondclaim}. Hence, we finish
the proof of Theorem~\ref{errorIth}. $\Box$





\end{document}